\documentclass{article}
\usepackage[utf8]{inputenc}
\usepackage{amsmath,amssymb,amsthm,graphicx}
\usepackage[all,cmtip]{xy}
\usepackage{scrextend}
\usepackage{color}

\usepackage{subfiles}

\usepackage{enumerate}

\def \Diff{{\mathop{\rm Diff\,}}}
\def \Hess{{\mathop{\rm Hess}}}

\def \ba{{\operatorname{\mathbf{a}}}}
\def \bb{{\operatorname{\mathbf{b}}}}

\def \bu{{\operatorname{\mathbf{u}}}}

\def \bx{{\operatorname{\mathbf{x}}}}
\def \by{{\operatorname{\mathbf{y}}}}
\def \bo{{\operatorname{\mathbf{0}}}}

\def \RR{\operatorname{\mathbb R}}

\def \CL{{\operatorname{\mathcal{L}}}}
\def \CV{{\operatorname{\mathcal{V}}}}
\def \CF{{\operatorname{\mathfrak{F}}}}
\def \CP{{\operatorname{\mathfrak{P}}}}

\def \dint{{\mathop{\frac{d }{d t}}}}

\newtheorem{thm}{Theorem}[section]
\newtheorem{defn}[thm]{Definition}
\newtheorem{lem}[thm]{Lemma}

\newtheorem{corollary}[thm]{Corollary}
\newtheorem{proposition}[thm]{Proposition}

\newtheorem{example}[thm]{Example}

\title{An Analytical Analogue of Morse's Lemma\footnote{This work is supervised by Prof. Dominic Joyce, and supported by the Engineering and Physical Sciences Research Council [EP/L015811/1].}}
\author{Yixuan Wang}
\date{}
\begin{document}
\maketitle

\begin{abstract}
\noindent The Morse function $f$ near a non-degenerate critical point $p$ is understood topologically, in the light of Morse's lemma. However, Morse's lemma standardizes the function $f$ itself, providing little information of how the gradient $\nabla f$ behaves.\\
In this paper, we prove an analytical analogue of Morse's lemma, showing that there exist smooth local coordinates on which a generic Morse gradient field $\nabla f$ near the critical point exhibits a unique linear vector field. We show that on a small neighbourhood of the critical point, the gradient field $\nabla f$ has a natural choice of standard form $V_0(\bx)=\sum_{i=1}^n \lambda_ix_i\frac{\partial}{\partial x_i}$, and this form only depend on the local behaviour of the Morse function and the Riemannian metric near the critical point. Then we present a constructive proof of the fact that given a generic Morse function $f$, for every critical point, there is a local coordinate on which the gradient field reduces to its standard form.
\end{abstract}
\tableofcontents


\section{The focus and main results.}
Given a Morse function $f$ on a Riemannian manifold $X$ with a critical point $p$, the Morse Lemma guarantees the existence of a coordinate chart centred at $p$, where the Morse function exhibits a standard formula determined only by the Morse index of $p$. The focus of this paper is to answer the following question: analogously, is there a local coordinate where the gradient of $f$ is standardized to a nice vector field?\\
As a theme, the question of finding the analytical friendly representations of vector fields via choices of coordinates orchestrates its variations in different chords. For example, Takens' work \cite{Takens74} concerns singularities of $C^k$ vector fields and their codimensions, and the proof utlizes general normal forms for these vector fields; Meanwhile, Palis and Smale \cite{PalaisSmale69} formulate the question implicitly in their theory of structural stability, for gradient dynamic systems satisfying what's now called Axiom A and strong transversality.\\
In the case of rectifying the Morse gradient field on a Riemannian manifold, we elaborate the answer to this question as follows, which is our main result.
\begin{defn}[$\mathbb{N}$-linearity condition]\label{defn_N_linearity}
We say that a group of real numbers $\lambda_1,\dots,\lambda_n$ satisfies the {\rm{$\mathbb{N}$-linearity condition}}, if the following equation
\[a_1\lambda_1+\dots+a_n\lambda_n-\lambda_i=0, \]
has no solution for all $(a_1,\dots,a_n)\in\mathbb{N}^n$ with $2\leq a_1+\dots+a_n$, and any $i=1,\dots,n$.
\end{defn}
\noindent As generically chosen $(\lambda_1,\dots,\lambda_n)\in(\RR/\{0\})^n$ satisfies the $\mathbb{N}$-linearity condition, this requirement is not restrictive.

\begin{thm}\label{thm_conjugate_function_existence}
Let $X$ be an n-dimensional Riemannian manifold equipped with Riemannian metric $g$ and a Morse function $f$, $V=\nabla f$, and $p$ be a critical point of $f$.
Let \emph{Morse eigenvalues} $\lambda_1,\dots,\lambda_n$ of $p$ be eigenvalues of $g^{jk}\nabla_iV_j|_p\in T_pX\otimes T^*_p X$. Assume that Morse eigenvalues of $p$ satisfy the $\mathbb{N}$-linearity condition.\\
Then there exists a local coordinate chart on a small neighbourhood $\tilde U$ containing $p$, which is a function $\Phi:U\rightarrow \tilde U$ with $U\subset \RR^n$ open and $\Phi(\bo)=p$, such that $\Phi$ is smooth with a smooth inverse on $U$, and
\[\Phi^*(V)=V_0=\sum_{i=1}^n \lambda_ix_i\frac{\partial}{\partial x_i},\]
for all $\bx\in U$. 
\end{thm}
\noindent{\bf Remark}:
\begin{itemize}
\item If we assume in addition that $\lambda_1\geq \lambda_2\geq \dots \geq \lambda_n$, then $\Phi$ is unique up to linear transformation
    \[\bx\mapsto (A_{ij})_{i,j=1}^n\bx,\]
    where $A_{ij}=0$ if $i\neq j$; and if the Morse eigenvalue $\lambda_i$ is of multiplicity $m$, with $\lambda_i=\dots=\lambda_{i+m-1}$, then $(A_{ij})_{i,j=i}^{i+m-1}\in \mathrm{GL}(m,\RR)$.
\item When the set of Morse eigenvalues $\{\lambda_1,\dots,\lambda_n\}$ fails to satisfy the $\mathbb{N}$-linearity condition, there exist some gradient fields $V$ with standard form $V_0$ that cannot be standardized to $V_0$. This is shown in Corollary \ref{corl_N-linearity_is_key}. In Example \ref{example_non_N_linearity}, we demonstrate this conclusion for the gradient field $V=2(x_1+x_2^2)\frac{\partial}{\partial x_1}+x_2\frac{\partial}{\partial x_1}$.
\end{itemize}
\section{The standard form of a Morse gradient field.}
In this section, we shall briefly recall Morse theory and the Morse Lemma, and give a rigours definition of the standard form of a Morse gradient field.
\subsection{Introduction to Morse theory.}
Let $X$ be a $n$-dimensional manifold equipped with Riemannian metric $g$.
\begin{defn}[Morse function]
For a smooth function
$f:X\rightarrow \RR$,
its \text{\rm{critical points}} are those points $p\in X$ where $\nabla f(p)=0$, and a critical point is called \text{\rm non-degenerate} if the matrix $\Hess f(p)$ is of full rank. A function $f$ is called a \text{\rm{Morse function}}, if it has only non-degenerate critical points.
\end{defn}
\noindent The existence of Morse functions is guaranteed, and in fact, Morse functions are dense in $C^\infty(X;\RR)$, as the following result from \cite[\S 6]{Milnor63} elaborates:
\begin{thm}[Abundance of Morse functions]
Any bounded smooth function $\hat f:X\rightarrow\RR$ can be uniformly approximated by a smooth Morse function $f$. Furthermore, $f$ can be chosen so that the i-th derivatives of $f$ on the compact set $K$ uniformly approximate the corresponding derivatives of $\hat f$, for all $i\leq k$.
\end{thm}
\noindent Let $f$ be a fixed Morse function. Its gradient $\nabla f$ is a well-defined vector field over $X$, and \textit{Morse flow} is the negative gradient flow induced by this vector field, which is the solution $G_\bx(t)$ to
\[\left\{ {\begin{array}{ll}
\dint G_\bx(t)=-\nabla f(G_\bx(t))\\
G_\bx(0)=\bx\\
\end{array} }, \right. \]
for any $\bx\in X$.
\begin{defn}[Index, stable and unstable manifold of a critical point]
Let $p$ be a critical point of Morse function $f$. Then the \text{\rm Morse index} of $p$ is the number of negative eigenvalues of  $\Hess f(p)$.\\
We define the \text{\rm stable manifold} of $p$ to be
\[W^s(p)={\bx\in X: \lim_{t\rightarrow +\infty}G_\bx(t)=p},\]
and the \text{\rm unstable manifold} of $p$ to be
\[W^u(p)={\bx\in X: \lim_{t\rightarrow -\infty}G_\bx(t)=p}.\]
\end{defn}
\noindent As the name suggests, $W^s(p)$ and $W^u(p)$ are indeed submanifolds of $X$. A detailed discussion along with proof can be found in \cite[\S2.1.d]{Audin14}, where a proof is given for so-called pseudo-gradient fields, vector fields that generalize $\nabla f$ and coincide with the gradient on a local chart containing the critical point. So here we just quote the result as follows.
\begin{thm}
The stable and unstable manifolds of the critical point $p$ are submanifolds of $X$ that are diffeomorphic to open disks. More over, we have
\[\dim W^u(p)=\mathrm{codim} W^s(p)=\mathrm{Index}(p).\]
\end{thm}
\noindent The gateway of Morse theory is the Morse lemma, without doubt. The lemma reveals the fact that for every critical point there exists a local chart where the Morse function is of standard form characterized by its Morse index, and as a result, it is justified to consider the index of a critical point as an invariant of local homeomorphisms. This lemma validates the definition of the Morse complex, a chain complex defined over the set of critical points ranked with respect to their indices, and the definition of the corresponding Morse homology. Some classical upshots of Morse homology are \cite[Chapter 4]{Audin14}: a) the alternating sum of the number of critical points of index $k$ (sign alters w.r.t. $k$) of a Morse function is a topological invariant --- the Euler characteristic; and b) the number of critical points of index $k$ is bounded below by the $k$-th Betti number of the manifold.\\
The original lemma was proven by Marston Morse in his paper \cite{Morse25}, 1925, using Gram--Schmidt orthogonalization method. Later on it was generalized to firstly suit calculus of variations in Hilbert spaces and then for Banach spaces in general, by Richard Palais \cite{Palais69}. The lemma also has a variation for \textit{Morse--Bott functions}, which are smooth functions with their critical loci being submanifolds of $X$ instead of isolated points, and their Hessian at a critical point are non-degenerate along the normal of the corresponding critical locus.\\
For the convenience of the reader, let us revisit the classical version of Morse lemma, and its proof, seen here mainly paraphrasing \cite[\S 2]{Milnor63}.
\begin{lem}[Morse Lemma]
Let $p$ be a non-degenerate critical point of Morse function $f$, and the index of $p$ is $k$. Then there is a local coordinate $(x_1,\dots,x_n)$ on a neighbourhood $U$ containing $p$, with $p$ associated with $\bo$. On this local coordinate, $\forall \bx\in U$,
\[f(\bx)=f(\bo)-x_1^2-\dots-x_k^2+x_{k+1}^2+\dots+x_n^2.\]
\end{lem}
\begin{proof}
Firstly, we claim that the following is true.
\begin{addmargin}[1em]{.5em}
\textbf{Claim}: Let $F$ be a $C^\infty$ function in a convex neighbourhood $V$ of $\bo$ in $\RR^n$, with $F(\bo)=0$. Then
\[F(\bx)=\sum_{i=1}^n x_ig_i(\bx)\]
for some $g_i(\bx)\in C^\infty(V;\RR)$, and $g_i(\bo)=\frac{\partial f}{\partial x_i}(\bo)$.\\
\textit{Proof of claim}. Note that
\[F(\bx)=\int_0^1 \frac{d F(tx_1,\dots,tx_n)}{d t}dt=\int_0^1 \sum_{i=1}^n \frac{\partial F(tx_1,\dots,tx_n)}{\partial x_i} x_i\, dt.\]
Therefore we can always let $g_i(\bx)=\int_0^1 \frac{\partial F(tx_1,\dots,tx_n)}{\partial x_1} dt$, and the claim follows.
\end{addmargin}
Apparently there exists a local coordinate where the critical point $p$ is mapped to $\bo$ in $\RR^n$, and we can assume that $f(\bo)=0$ on this local chart. Applying aforementioned claim to $f$ yields
\[f(\bx)=\sum_{i=1}^nx_ig_i(\bx)\]
for $\bx$ in some neighbourhood of $\bo$. Since $\bo$ is the critical point,
\[g_i(\bo)=\frac{\partial f(\bo)}{\partial x_i}=0.\]
Once again, applying the claim to $g_i(\bx)$, we have
\[g_i(\bx)=\sum_{j=1}^n x_j h_{ij}(\bx),\]
for some smooth functions $h_{ij}$. So
\begin{equation}\label{eqn_non_linear_Hess_f}
f(\bx)=\sum_{i,j=1}^nx_ix_jh_{ij}(\bx).
\end{equation}
We assume that $h_{ij}(\bx)=h_{ji}(\bx)$, as taking $\hat h_{ij}=\hat h_{ji}=\frac{1}{2} (h_{ij}+h_{ji})$ reduces the case to our assumption. Moreover, the matrix $(h_{ij}(\bo))=(\frac{1}{2}\frac{\partial^2 f(\bo)}{\partial x_i \partial x_j})$, and hence it is non-singular.\\
The coordinate where $f$ is of the desired form, possibly on a smaller neighbourhood $U_1$ of $\bo$, is constructed inductively. Suppose there exists a coordinate $(u_1,\dots,u_n)$ on $U_1$ where
\[f=\pm u_1^2 \pm \dots \pm u_{r-1}^2+\sum_{i,j\geq r} u_iu_jH_{ij}(u_1,\dots,u_n),\]
with the matrix $(H_{ij})$ symmetric. After a linear exchange in the last $n-r+1$ coordinates, we may assume that $H_{rr}(\bo)\neq 0$. Now introduce the new coordinate $(v_1,\dots,v_n)$ as
\begin{align}
\nonumber v_i&=u_i\text{ for }i\neq r,\\
\nonumber v_r(\bu)&=\sqrt{|H_{rr}(\bu)|}\left[u_r+ \sum_{i>r} u_i H_{ir}(\bu)/H_{rr}(\bu)\right].
\end{align}
On a small enough neighbourhood $U_2\subset U_1$ of the origin, $(v_1,\dots,v_n)$ is a smooth transformation with a smooth inverse, hence it will serve as a coordinate on a sufficiently small neighbourhood $U_3$. It is easy to verify that
\[f=\sum_{i\leq r}\pm v_i^2+\sum_{i,j >r} v_iv_j H_{ij}'(v_1,\dots,v_n),\]
which completes the induction, and proves this lemma.
\end{proof}
\subsection{First observations, and the standard form of a gradient field.}\label{subsec_standard_form}
\noindent The Morse lemma offers a constructive way of finding a local coordinated standardizing the Morse function. However, it is less helpful when the Morse gradient field is of our concern. In fact,
\begin{equation}\label{eqn_nabla_f}
\nabla f(\bx)=\sum_{i,j=1}^n g^{ij}\frac{\partial f}{\partial x_j}\frac{\partial}{\partial x_i},
\end{equation}
where $g^{ij}=(g^{-1})_{ij}$. So the local vector field may appear fully non-linear, even if $f$ is standard in the sense of the lemma.\\
What do we expect for the Morse vector field on a local coordinate near the critical point? A first look yields that we can always set the first order terms of $\nabla f$ to be diagonal, and the coefficient of this diagonal form is coordinate independent:
\begin{proposition}\label{prop_local_form_with_error}
There exists a local coordinate $(x_1,\dots,x_n)$ on a local neighbourhood $U$ containing the critical point $p$ where $p$ is mapped to $\bo$, and
\[\nabla f(\bx)=\sum_{i=1}^n \lambda_ix_i\frac{\partial}{\partial x_i}+O(\|\bx\|^2)\frac{\partial}{\partial x_i},\]
where $\lambda_1,\dots,\lambda_k<0$, $\lambda_{k+1},\dots,\lambda_n>0$, with $k$ the Morse index of $p$.\\
 Furthermore, the set of eigenvalues $\{\lambda_1,\dots,\lambda_n\}$ is independent of the choice of coordinates.
\end{proposition}
\begin{proof}
A closer look at equation \eqref{eqn_non_linear_Hess_f} tells us that
\[f(\bx)=\sum_{i,j=1}^nx_ix_jh_{ij}(\bx), \text{ with }h(\bo)=\frac{1}{2}\Hess f(\bo),\text{ where }h(\bx):=(h_{ij}(\bx)).\]
The Hessian as well as the metric are symmetric bilinear forms over $T_\bx X$. Let $B:T_\bx X\times T_\bx X\rightarrow \RR$ be such a bilinear form, then a coordinate transformation of the local coordinate $\bx\mapsto \by$ induces a congruence of the matrix of $B$ w.r.t. Jacobian of the coordinate change, i.e.
\begin{align}
\nonumber B(\by(\bx))=J^T B(\bx) J&, \text{ where } J=\left(\frac{\partial y_i(\bx)}{\partial x_j}\right),\\
\nonumber B^{-1}(\by(\bx))=J^{-1}B^{-1}(\bx)(J^{-1})^T&, \text{ when }J\text{ is invertible}.
\end{align}
Consequently, an invertible linear transformation $\by= A\bx$ gives the matrix similarity
\[g^{-1}|_{\by=\bo}\,\Hess f|_{\by=\bo}=A^{-1} (g^{-1}|_{\bx=\bo}\,\Hess f|_{\bx=\bo})A.\]
Let $A$ be the positive definite matrix that diagonalizes both $g^{-1}(\bo)$ and $\Hess f(\bo)$, such that $g^{-1}|_{\by=\bo}\,\Hess f|_{\by=\bo}=\mathrm{diag}\, \{\lambda_1,\dots,\lambda_n\}$. The existence of $A$ is guaranteed by the fact that two symmetric matrices, with one of which being positive definite, can be diagonalized simultaneously. On this new coordinate,
\begin{equation}\label{eqn_g_inverse_h_matrix}
(g^{-1}(\by)h(\by))_{ij}=\sum_{k=1}^n g^{ik}(\by)h_{kj}(\by)= \frac{1}{2}\lambda_i\delta_{ij}+O(\|\by\|).
\end{equation}
Equation \eqref{eqn_non_linear_Hess_f} gives
\begin{align}
\nonumber \frac{\partial f}{\partial y_j}&=\sum_{k=1}^n 2y_kh_{jk}(\by)+\sum_{k,l=1}^n y_ky_l\frac{\partial h_{kl}(\by)}{\partial y_j}
\end{align}
 and combining this with \eqref{eqn_nabla_f} and \eqref{eqn_g_inverse_h_matrix},
\begin{align}
\nonumber \nabla f(\by)&=\sum_{i,j=1}^n \left(g^{ij}\sum_{k=1}^n 2y_kh_{jk}(\by)+O(\|\by\|^2)\right)\frac{\partial }{\partial y_i}\\
\nonumber &=\sum_{i,k=1}^n \left(2y_k\sum_{j=1}^n g^{ij}h_{jk}(\by)+O(\|\by\|^2)\right)\frac{\partial }{\partial y_i}=\sum_{i=1}^n \lambda_iy_i\frac{\partial}{\partial y_i}+O(\|\by\|^2)\frac{\partial}{\partial y_i}.
\end{align}
As signs are preserved during our construction, we may assume that $\lambda_1,\dots,\lambda_k<0$ and $\lambda_{k+1},\dots,\lambda_n>0$, where $k$ is the Morse index of $p$.\\
Furthermore, for the vector field $V(\bx)=\nabla f(\bx)\in C^\infty(T_\bx X)$, the Hessian $\Hess f(\bx)=\nabla V\in C^\infty (T^*_\bx X\otimes T_\bx X)$. This indicates that $\{\lambda_1,\dots,\lambda_n\}$ is the set of eigenvalues of $g^{-1}\Hess f|_p$, and it is invariant under local diffeomorphisms. Or this can be viewed directly from our proof: a local diffeomorphism $\phi$ induces matrix similarity for the matrix of coefficients of the leading order term of $V$, with its Jacobi $J\phi(\bo)$, which apparently doesn't alter the set of eigenvalues.
Thus the proposition is proven.
\end{proof}

\noindent And we give the linear part of $\nabla f$ a name.
\begin{defn}[Standard form]
For a critical point $p$ of a Morse function $f$, there exists a local coordinate over a neighbourhood $U$ containing $p$, such that
\[\nabla f(\bx)=\sum_{i=1}^n \lambda_ix_i\frac{\partial}{\partial x_i}+O(\|\bx\|^2)\frac{\partial}{\partial x_i},\forall \bx\in U.\]
The linear part of $\nabla f$ is called the {\rm{standard form of $\nabla f$}}, denoted by
\[V_0(\bx)=\sum_{i=1}^n \lambda_ix_i\frac{\partial}{\partial x_i}.\]\\
The set of eigenvalues $\{\lambda_1,\dots,\lambda_n\}$ is coordinate independent, and they will be called {\rm Morse eigenvalues} of the critical point $p$.
\end{defn}
\noindent
The discussion above leads to the natural question: Given a Morse function $f$ on a Riemannian manifold $X$, is there a local coordinate system on a neighbourhood $U$ containing the critical point $p$, with $p$ to the origin, such that $\nabla f$ reduces to its standard form
throughout $U$?\\
In general, the answer is: yes, if the problem is generically posed, then such local coordinates exist. For the rest of the paper, we will see what ``generic'' means in strict, analytical sense, and give a constructive proof of such a local coordinate in this case.
\section{Proof of main results.}
Here is a roadmap to help the reader navigate the proof of our results:
\begin{enumerate}[Step 1:]
\item \S \ref{subsubsec_3.1.1}. Under the $\mathbb{N}$-linearity assumption, we find a local coordinate $(x_1,\dots,x_n)$ where the general vector field differs from its standard form only by a locally flat function, namely $|V-V_0|(\bx)=O(|\bx|^\infty)$. This is shown in Proposition \ref{prop_lem_standardizing_on_WuWs}. Moreover, when $\mathbb{N}$-linearity fails, Corollary \ref{corl_N-linearity_is_key} proves that there exists a gradient field $V$ that cannot be standardized, illustrated by Example \ref{example_non_N_linearity}, and both demonstrate that the restriction of $\mathbb{N}$-linearity condition is sufficient and almost necessary in this process.
\item \S\ref{subsubsec3.1.2}. There exist local coordinates where the unstable manifold corresponds to $(x_1,\dots,x_k,0,\dots,0)$, the stable manifold $(0,\dots,0,x_{k+1},$ $\dots,x_n)$; furthermore, the relation $|V-V_0|=O(|\bx|^\infty)$ is preserved. This is Proposition \ref{lem_coordinate_out_of_WuWs}.
\item Based on the first two results, we modify the choice of local coordinates further, and construct a coordinate chart where $V=V_0$ on both the unstable submanifold, which corresponds to $(\bx_k,\bo)$, and the stable submanifold, which is $(\bo,\bx_{n-k})$; on top of that, $|V-V_0|=O(|\bx|^\infty)$ is preserved within a small neighbourhood of the critical point. This is proven in Proposition \ref{lem_v=v_0atWuWs}.
\item \S \ref{subsubsec3.1.3}. In Theorem \ref{thm_estimate_V-V_0_WuWs}, the final estimate of this subsection \S \ref{subsec_2.1} is given, and it will prove to be crucial later: on a small neighbourhood containing $p$, there exists a coordinate chart such that for every large enough $\alpha$ and some positive constants $C=C(\alpha)$, we have $|V-V_0|\leq C|\bx_k|^\alpha|\bx_{n-k}|^\alpha$.
\item In \S \ref{subsubsec3.2.2}, we present an operator $\CF: U\times L^p_{k,\delta}((-\infty,0];\RR^n)\rightarrow L^p_{k,\delta}((-\infty,0];$ $\RR^n)$ with its fixed point closely related to the local diffeomorphism $\Phi$ of our concern. The norm of operator $\CF$ relies closely on the weighted Sobolev spaces which it is defined on, as explained in Proposition \ref{lem_CF_shrink_in_range} and Lemma \ref{lem_integration_with_delta}; also, the operator $\CF$ moves the zero function in a controlled manner, which is discussed in Proposition \ref{lem_CF_drifting_domain}.
\item By carefully choosing regularity and weight of weighted Sobolev spaces, we find a convex region of the function space where $\CF$ is a contraction operator. This is shown in Theorem \ref{thm_CF_is_contracting_on_Omega}.
\item In \S \ref{subsubsec_proof_of_main}, the local diffeomorphism $\Phi$ is constructed from the fixed point of $\CF$, and its uniqueness and regularity are validated by a Banach space version of the Implicit Function Theorem. And we complete the proof of Theorem \ref{thm_conjugate_function_existence}.
\end{enumerate}
\subsection{Standardizing the vector field with controlled errors.}\label{subsec_2.1}
In this section, we will standardize the vector field of the gradient by finding suitable local coordinates centred at the the critical point, on which the gradient reduces to a form which is reasonably close to its standard form.
\subsubsection{Standardising generic $(V-V_0)$ to flat functions.}
\label{subsubsec_3.1.1}
To begin with, we will establish that we can find submanifold diffeomorphisms such that a general gradient field is arbitrarily close to its standard form.
\begin{proposition}\label{prop_lem_standardizing_on_WuWs}
Let $V(\bx)=\nabla f(\bx)=\sum_{i=1}^n (\lambda_ix_i+O(\|\bx\|^2))\frac{\partial}{\partial x_i}$ and $V_0$ be the standard form of $V$. Assume that $\lambda_1,\dots,\lambda_n$ satisfy the $\mathbb{N}$-linearity condition. Then there exists a local coordinate chart $(x_1,\dots,x_n)$, with the critical point $p$ mapped to $\bo$, such that
\[V(\bx)=V_0(\bx)+\sum_{i=1}^n O(|\bx|^\infty)\frac{\partial}{\partial x_i}.\]
%
\end{proposition}
\begin{proof}
Let $(x_1,\dots,x_n)$ be a local coordinate chart on a neighbourhood of the critical point $p$, with the critical point mapped to $\bo$.\\
We now work with \emph{germs} near $\bo$ on $\RR^n$. A germ of functions at $\bo$ in $\RR^n$, denoted by $C^\infty(\RR^n)_\bo$, is the equivalence class of functions that are identical near $\bo$: Let $(U,f)$ be a pair, where $\bo\in U\subset\mathbb{R}^n$, $U$ open, and function $f:U\rightarrow \RR$ be smooth; then two such pairs $(U_1,f_1)$ and $(U_2,f_2)$ are equivalent if there exists open neighbourhood $U_3\subset U_1\cap U_2$ with $\bo\in U_3$, such that $f_1|_{U_3}=f_2|_{U_3}$.\\
\emph{The germ of diffeomorphisms of $\RR^n$ fixing $\bo$} is defined with the equivalence relation as follows: For triplets of the form $(U,V,\phi)$, where $U,V\subset \RR^n$ are open with $\bo\in U,V$, and $\phi:U\rightarrow V$ is a diffeomorphism with $\phi(\bo)=\bo$, we say that $(U_1,V_1,\phi_1)$ and $(U_2,V_2,\phi_2)$ are equivalent if there exists an open set $U_3\subset U_1\cap U_2$, $\bo\in U_3$, such that $\phi_1|_{U_3}=\phi_2|_{U_3}$. The germ of smooth diffeomorphisms of $\RR^n$ fixing the origin, denoted by $\Diff (\RR^n)_\bo$, is an infinite dimensional Lie group. The Lie group structure of this group of germs and its Lie algebra are studied by Robart and Kamran in \cite[Theorem 3]{Robart97}.\\
The group of germs of $C^\infty$ diffeomorphisms fixing the origin, denoted by $\Diff(\RR^n)_\bo$, has nested normal subgroups $\Diff(\RR^n)_\bo^l, l\geq 2$. Each $\Diff(\RR^n)_\bo^l$ consists of those diffeomorphisms of the form $\sigma_l(\bx)= Id(\bx) + p_l(\bx)$, where $p_l:\RR^n\rightarrow \RR^n$ is a $n$-vector of polynomials of order bigger than or equal to $l$. Then these infinite dimensional normal subgroups are nested as $\Diff(\RR^n)_\bo^2\supset \Diff(\RR^n)_\bo^3 \supset \dots$. In addition, for $m<l$, the quotient of subgroups can be defined $\Diff(\RR^n)_\bo^{m,l}=\Diff(\RR^n)_\bo^m/\Diff(\RR^n)_\bo^l$. This quotient space is a finite dimensional Lie group, with dimension,
\[\dim \Diff(\RR^n)_\bo^{m,l}=n\left(
\left(
  \begin{array}{c}
    n+m-1 \\
    m \\
  \end{array}
\right)
+\dots +
\left(
  \begin{array}{c}
    n+l-2 \\
    l-1 \\
  \end{array}
\right)
\right).\]
As the error of $V$ compared to its standard form $V_0$ is of at least second order, we will work with $\Diff (\RR^n)_\bo^{2,l}$, $l>2$.\\
In the same manner, we define the vector space $\CV_{\lambda_1,\dots,\lambda_n}$ of germs of the vector fields of the same standard form $V_0$ as the equivalent class of vector fields that agrees on an open neighbourhood containing the origin, namely,
\[\CV_{\lambda_1,\dots,\lambda_n}=\left\{v=V_0+\sum_{i=1}^n O(\|\bx\|^2)\frac{\partial}{\partial x_i}\right\}/\sim\]
 where $v\sim w$ if there exists an open neighbourhood $U$ of $\bo$ such that $v|_U=w|_U$. Apparently, $\CV_{\lambda_1,\dots,\lambda_n}$ is an infinite dimensional vector space with its origin being the germ of $V_0$.\\
Germs of vector fields in the vector space $\CV_{\lambda_1,\dots,\lambda_n}$ that are identical up to rank $l$ form another vector space,
\[\CV_{\lambda_1,\dots,\lambda_n}^l=\CV_{\lambda_1,\dots,\lambda_n}/\stackrel{l}{\approx},\]
where for $v,w \in \CV_{\lambda_1,\dots,\lambda_n}$, $v\stackrel{l}{\approx} w$ iff $v/O(|\bx|^{l})=w/O(|\bx|^{l})$.\\
Vector space $\CV_{\lambda_1,\dots,\lambda_n}^l$ is finite dimensional, and it is easy to see that $\dim \CV_{\lambda_1,\dots,\lambda_n}^l$ $ = \dim \Diff (\RR^n)_\bo^{2,l}$.\\
The action of the diffeomorphism group $\Diff(\RR^n)_\bo^{2,l}$ on $\CV_{\lambda_1,\dots,\lambda_n}^l$ is well-defined. The lemma is proven if for all $l$, vector field $V$ is in the orbit of the standard form $V_0$ up to rank $l$, namely if the orbit contains $w_l=\frac{V}{O(|\bx|^l)}$. If this is true for some $l$, then we get the existence of a diffeomorphism $\psi_l$ that ``straightens'' $V$ up to order $l$, namely $\psi_l^*(w_l)=V_0$.\\
The vector field $w_l$ is in the orbit of $V_0$ if $\Diff(\RR^n)_\bo^{2,l}$ acts transitively on $\CV_{\lambda_1,\dots,\lambda_n}^l$. In other words, let $G$ be the subgroup that fixes $V_0$, then we claim that the orbit of $V_0$ is the whole set of $\CV_{\lambda_1,\dots,\lambda_n}^l$ if and only if $G$ is trivial.
\begin{addmargin}[1em]{.5em}
Here is why this claim is true. \\
Firstly, that the stabilizer $G$ is trivial is equivalent to $\dim G=0$, and is equivalent to that the orbit of $V_0$, $\mathcal{O}(V_0)$, is open in $\CV_{\lambda_1,\dots,\lambda_n}^l$. This is because as a submanifold of $\CV_{\lambda_1,\dots,\lambda_n}^l$, the orbit is of the same dimension as the ambient manifold, $\dim \mathcal{O}(V_0)=\dim \CV_{\lambda_1,\dots,\lambda_n}^l-\dim G$.\\
Secondly, that the orbit $\mathcal{O}(V_0)$ is open is equivalent to $\mathcal{O}(V_0)$ being the whole space $\CV_{\lambda_1,\dots,\lambda_n}^l$. This is because for any $\hat w\in \CV_{\lambda_1,\dots,\lambda_n}^l$, which is represented by $\hat w= V_0 + \sum_i Q_i(\bx)\frac{\partial}{\partial x_i} \in \CV_{\lambda_1,\dots,\lambda_n}$ with each polynomial $Q_i(\bx)$ of rank no larger $l$, the dilation $\sigma_\epsilon:\bx\mapsto \epsilon \bx$ leaves $V_0$ invariant and reduces $\hat w-V_0$ non-linearly by at least a factor of $\epsilon$, as $Q_i(\bx)$ are at least quadric in $\bx$. For every open neighbourhood of $V_0$, there exists a small enough $\epsilon$ for $\hat w$, such that $\sigma_\epsilon^*(\hat w)$ is in that neighbourhood; and as $\mathcal{O}(V_0)$ is an open neighbourhood of $V_0$, there exist an $\epsilon$ and a diffeomorphism $\phi$ such that $V_0= \phi^*(\sigma_\epsilon^*(\hat w))=(\sigma_\epsilon\circ\phi)^*(\hat w)$. Consequently, $\hat w\in \mathcal{O}(V_0)$, and $\mathcal{O}(V_0)=\CV_{\lambda_1,\dots,\lambda_n}^l$.
\end{addmargin}
As $\Diff(\RR^n)_\bo^{2,l}$ is connected, that $G$ is trivial is equivalent to the Lie algebra $\mathfrak{g}$ of $G$ being trivial. This Lie algebra is characterized by
\[\mathfrak{g}=\{w\in \text{\bf diff}(\RR^n)^{2,l}_\bo:[V_0,w]=0\}.\]
Using the fact that the basis of the Lie algebra $\text{\bf diff}(\RR^n)_0^{2,l}$ is $x_1^{a_1}\dots x_n^{a_n}\frac{\partial}{\partial x_i}, 2\leq a_1+\dots +a_n<l$,
\[\left[\sum_j \lambda_jx_j\frac{\partial}{\partial x_j},x_1^{a_1}\dots x_n^{a_n}\frac{\partial}{\partial x_i}\right]=(a_1\lambda_1+\dots+a_n\lambda_n-\lambda_i)x_1^{a_1}\dots x_n^{a_n}\frac{\partial}{\partial x_i}.\]
As a result, $\mathfrak{g}$ is trivial if $a_1\lambda_1+\dots + a_n\lambda_n-\lambda_i\neq 0$ for all $i=1,\dots,n$ and $a_1,\dots,a_n\in \mathbb{N}$ with $2\leq a_1+\dots+a_n<l$. This gives the \emph{$\mathbb{N}$-linearity condition} for $\lambda_i$'s at the $l$-th step.\\
If we require the $\mathbb{N}$-linearity condition to hold for all $l>2$, which is not restrictive as such choices of Morse functions are still generic, then for every $l$ there exists a diffeomorphism $\psi_l$ such that $\psi_l^*(w_l)=V_0$. In other words, $|\psi_l^*(V)-V_0|(\bx)=O(|\bx|^l)$.
\newline
A closer look at the derivation of $\psi_l$ yields that
\[\psi_{l+1}=\psi_l+(\text{homogeneous term of order } (l+1)).\]
Hence $\psi=\lim_{l\rightarrow \infty}\psi_l$ is a formal power series with $\psi_l$ being its first $l$ terms, and formally $\psi^*(V)=V_0$. Consequently, there exists a formal sequence of diffeomorphisms that ``standardizes'' $V$ up to an error of order $\infty$.\\
In fact, we can choose a local diffeomorphism that has exactly this formal sequence, with the help of the \emph{Borel theorem} \cite[Theorem I.1.3]{Moerdijk13}:
\begin{addmargin}[1em]{.5em}
Let $C^\infty(\RR^{n+m})$ be the ring of smooth functions over $\RR^{n+m}$, and $\mathbf{m}^\infty_{\RR^n\times \{\bo\}}$ be the ideal of functions which are flat on $\RR^n \times \{\bo\}$ (flat means all partial derivatives vanish at $\RR^n \times \{\bo\}$). Let $C^\infty(\RR^n)[[y_1,\dots,y_m]]$ be the ring of formal power series of $y_1,\dots,y_n$ with smooth coefficients in $C^\infty(\RR^n)$. Then the Taylor series gives an isomorphism
\[C^\infty(\RR^{n+m})/\mathbf{m}^\infty_{\RR^n\times \{\bo\}}\xrightarrow{\sim} C^\infty(\RR^n)[[y_1,\dots,y_m]].\]
\end{addmargin}
Then for the constructed formal power series $\psi$, there exists a (non-unique) smooth function $\Psi:\RR^n\rightarrow\RR^n$, such that the Taylor series of $\Psi$ at $\bo$ is $\psi$, and $\Psi$ is a local diffeomorphism near $\bo$, as $D\Psi(\bo)=Id$ is invertible. Moreover, $\Psi^*(V)=V_0+O(|\bx|^{\infty})$, where $O(|\bx|^{\infty})$ is a locally flat function.
\end{proof}

\begin{corollary}\label{corl_N-linearity_is_key}
When $\mathbb{N}$-linearity condition fails to hold for the set of eigenvalues $\{\lambda_1,\dots,\lambda_n\}$, there exists a Morse function $f$ whose gradient cannot be standardized with the method of Proposition \ref{prop_lem_standardizing_on_WuWs}.
\end{corollary}
\begin{proof}
Assume that there exists an $i$ and a set of non-negative integers $(a_1,\dots,a_n)$ with $a_1+\dots a_n\geq 2$ such that
\[a_1\lambda_1+\dots+a_n\lambda_n=\lambda_i.\]
Then we claim that the stabilizer $G$ of $V_0$ in the proof of Proposition \ref{prop_lem_standardizing_on_WuWs} is non-trivial.
\begin{addmargin}[1em]{.5em}
Let $\phi: \RR^n\rightarrow \RR^n$ be
\[x_i\mapsto x_i+x_1^{a_1}\dots x_n^{a_n},\text{ and } x_j\mapsto x_j, j\neq i.\]
Then $\phi$ is a local diffeomorphism near the origin. Moreover, by the chain rule, that $\phi$ is a stabilizer is equivalent to
\begin{align}
\nonumber \sum_j \lambda_j x_j \frac{\partial}{\partial x_j}&= \sum_j\left( \lambda_j x_j \left( \sum_l \frac{\partial \phi_l}{\partial x_j}\frac{\partial}{\partial \phi_l} \right)\right)\\
\nonumber &= \sum_l\left(\sum_{j}\left(\lambda_j x_j \frac{\partial}{\partial x_j}\right)\phi_l\right)\frac{\partial}{\partial \phi_l}\\
\nonumber &=\sum_l \lambda_l \phi_l \frac{\partial}{\partial \phi_l}.
\end{align}
So $\phi$ is a stabilizer of $V_0$, if and only if
\[\sum_{j}\left(\lambda_j x_j \frac{\partial}{\partial x_j}\right)\phi_l=\lambda_l\phi_l,\forall l=1,\dots,n.\]
It is easy to check that the $\phi$ we proposed earlier satisfies this relation, hence it is a non-trivial element of the stabilizer $G$ of $V_0$.
\end{addmargin}
As a result, the dimension argument $\dim \mathcal{O}(V_0)=\dim \CV_{\lambda_1,\dots,\lambda_n}^l-\dim G$ reveals that the orbit $\mathcal{O}(V_0)$ is a genuine subgroup of $\CV_{\lambda_1,\dots,\lambda_n}^l$, so long as $l\geq a_1+\dots+a_n$. By choosing $\nabla f$ such that $w_l=\nabla f/O(\bx^l)\in \CV_{\lambda_1,\dots,\lambda_n}^l -\mathcal{O}(V_0)$, we find a collection of Morse functions that cannot be standardized by a formal power series, as in Proposition \ref{prop_lem_standardizing_on_WuWs}, which concludes our proof.
\end{proof}
\begin{example}
\label{example_non_N_linearity}
Let 
\[V=2(x_1+x_2^2)\frac{\partial}{\partial x_1}+x_2\frac{\partial}{\partial x_1},\]
 and its standard form $V_0=2x_1\frac{\partial}{\partial x_1}+x_2\frac{\partial}{\partial x_1}$, then there exists no local diffeomorphism $\phi$ of $\RR^n$ fixing $\bo$ such that 
$\phi^*(V)=V_0.$
\end{example}
\begin{proof}
To prove that $V$ is not in the orbit of $V_0$ and consequently such a diffeomorphism $\phi$ does not exist, it is enough to prove that $V$ is not in the orbit generated by $\Diff (\RR^n)_\bo^{2,2}$ acting on $V_0$. Let $\psi\in \Diff (\RR^n)_\bo^{2,2}$, 
\[\psi(\bx)_1=y_1=x_1+ax_2+bx_1x_2+cx_1^2+dx_2^2,\]
\[\psi(\bx)_2=y_2=x_2+a' x_1+b'x_1x_2+c'x_1^2+d'x_2^2,\]
then 
\[V_0(\bx)=2x_1\frac{\partial}{\partial x_1}+x_2\frac{\partial}{\partial x_1}=2(y_1+y_2^2)\frac{\partial}{\partial y_1}+y_2\frac{\partial}{\partial y_1}=V(\by)\]
if and only if
\[\left(2x_1\frac{\partial}{\partial x_1}+x_2\frac{\partial}{\partial x_1}\right)y_1=2(y_1+y_2^2), \, \left(2x_1\frac{\partial}{\partial x_1}+x_2\frac{\partial}{\partial x_1}\right)y_2=y_2.\]
This set of equation has solution $y_2=x_2$ for $\psi(\bx)_2$. However, for $\psi(\bx)_1$, its coefficient $d$ has to satisfy $d=d+1$, which has no solution. As a result, $V$ is not in the orbit of $V_0$, so cannot be standardised.
\end{proof}
\subsubsection{$V=V_0$ on stable and unstable submanifolds.}\label{subsubsec3.1.2}
\noindent Now we show that it is possible to build coordinates out of the stable and unstable loci near a critical point, while maintaining the flatness of the difference between the vector field $V$ and its standard counterpart $V_0$. It is useful to recall that the stable and unstable loci of the critical point $p$ are in fact submanifolds, intersecting transversally at only $p$.
\begin{proposition}\label{lem_coordinate_out_of_WuWs}
Let $p$ be a critical point,  and $U$ a small open neighbourhood containing $p$. Let $W^u$ be the $k$-dimensional unstable submanifold of $p$ contained in $U$, and $W^s$ be the $(n-k)$-dimensional stable submanifold in $U$. Given that $\lambda_1,\dots,\lambda_n$ are $\mathbb{N}$-linearly independent, there exists a local coordinate chart $\Phi: \hat U \rightarrow U$, where $\hat U\subset \RR^n$ is an open neighbourhood of $\bo$, such that $\Phi(\bo)=p$, and
\[\Phi\left((\RR^k\times \{\bo_{n-k}\})\cap \hat U\right)=W^u,\]
\[\Phi\left((\{\bo_{k}\}\times\RR^{n-k} )\cap \hat U\right)=W^s.\]
In addition, $|V-V_0|(\bx)=O(|\bx|^\infty)$ for all $\bx\in \hat U$.
\end{proposition}
\begin{proof}
Utilizing Proposition \ref{prop_lem_standardizing_on_WuWs}, let $(\hat U_1,\phi)$ be the coordinate chart covering $p$, with $\hat U_1\subset \RR^n$ open, $\phi:\hat U_1\rightarrow U$ and $\phi(\bo)=p$, such that the estimate $|V-V_0|(\bx)=O(|\bx|^\infty)$ holds. Now, $\phi^{-1}(W^u)$ and $\phi^{-1}(W^s)$ are submanifolds of dimension $k$ and $n-k$ in $\hat U_1$, and they transversally intersect at only $\bo$. Hence there exist charts $(\hat U^k_1,\psi^u)$, $(\hat U^{n-k}_1,\psi^s)$, where $\hat U^k_1\subset \RR^k$ is an open neighbourhood of $\bo_k$ and $\hat U^{n-k}_1 \subset \RR^{n-k}$ is an open neighbourhood of $\bo_{n-k}$, s.t. $\psi^u(\bo_{k})=\bo$, $\psi^s(\bo_{n-k})=\bo$, and
\[\psi^u:\hat U^k_1\longrightarrow\phi^{-1}(W^u)\hookrightarrow\hat U_1, \psi^u(\bx_k)=(\bx_k,Y(\bx_k)),\]
\[\psi^s:\hat U^{n-k}_1\longrightarrow\phi^{-1}(W^s)\hookrightarrow\hat U_1, \psi^s(x_{n-k})=(Z(\bx_{n-k}),\bx_{n-k}),\]
where $Y:\hat U_1^k\rightarrow \RR^{n-k}$, and $Z:\hat U_1^{n-k}\rightarrow \RR^k$ are smooth functions, as they specify the embedding of corresponding submanifolds $\phi^{-1}(W^u)$ and $\phi^{-1}(W^s)$ in $\hat U_1$. Moreover, we claim that $Y$ and $Z$ are locally flat, namely $Y(\bx_k)=O(|\bx_k|^\infty)$, and $Z(\bx_{n-k})=O(|\bx_{n-k}|^\infty)$.
\begin{addmargin}[1em]{.5em}
Let us justify this claim for the unstable submanifold on the coordinate chart. Denote
\[V(\bx)=\nabla f(\bx)=\sum_{i=1}^n (\lambda_i x_i+O_i(\bx))\frac{\partial}{\partial x_i}, \]
where $O_i(\bx)=O(|\bx|^\infty)$ for all $i=1,\dots,n$.
Then flow lines on the unstable submanifold parametrized by $\psi^u$, consists exactly of those $\bu(t)$ that are the solutions to
\[\left\{ {\begin{array}{ll}
\dint \bu(t)=V(\bu(t))_i=\lambda_iu_i(t)+O_i(\bu(t))\\
\bu(0)=(u_1(0),\dots,u_n(0))\\
\end{array} }, \right. \]
and the implicit solution $\bu(t)=(x_1(t),\dots,x_k(t), Y_1(\bx_k(t)),\dots, Y_{n-k}(\bx_k(t)))$ is
\begin{align*}
&x_i(t)=u_i(0)e^{\lambda_i t}+e^{\lambda_i t}\int_0^t e^{-\lambda_i s}O_i(\bx_k(s),Y(\bx_k(s)))ds, i=1,\dots,k,\\
&Y_j(\bx_k(t))=u_j(0)e^{\lambda_j t}+e^{\lambda_j t}\int_0^t e^{-\lambda_j s}O_j(\bx_k(s),Y(\bx_k(s)))ds, j=1,\dots,n-k.
\end{align*}
For any multi-index $\ba=(a_1,\dots,a_k)\in \mathbb{N}^k$, the $\ba$-th derivative of $Y_j(\bx_k)$ at $\bx_k=\bo_k$ is
\begin{align*}
\left(\frac{\partial}{\partial \bx_k}\right)^\ba \left. Y_j(\bx_k)\right|_{\bx=\bo_k}&=\left(\frac{\partial}{\partial \bx_k(t)}\right)^\ba \left. Y_j\left(\bx_k(t)\right)\right|_{\bx(t)=\bo_k}\\
&= \left(\frac{\partial}{\partial \bx_k(t)}\right)^\ba \left. O_j(\bx_k(t), Y(\bx_k(t)))\right|_{\bx(t)=\bo_k}=0,
\end{align*}
 because $\psi^u(\bo_k)=\bo$ requires $(\bx_k(t), Y(\bx_k(t))|_{\bx(t)=\bo_k}=\bo$, and all derivatives of $O_i$ vanishes at $\bo$. Hence, $Y(\bx_k)$ attributes its local flatness to that of $O_i(\bx)$'s.\\
The same conclusion can be drawn following a similar argument for the stable submanifold and diffeomorphism $\psi^s$.
\end{addmargin}
Now let us construct an $\RR^n$ diffeomorphism,
\begin{align}
\nonumber \psi: \hat U &\longrightarrow \hat U_2\\
\nonumber \bx &\mapsto \psi^u(\bx_k)+ \psi^s(\bx_{n-k}),
\end{align}
where $\hat U=\left(\hat U_1^k\times\hat U_1^{n-k}\right)\cap \psi^{-1}(\hat U_2)$ for some $\hat U_2\subset \hat U_1$ an open neighbourhood of $\bo$. For simplicity, we still denote $\left(\hat U_1^k\times\{\bo_{n-k}\}\right)\cap \hat U$ by $\hat U_1^k\times \{\bo_{n-k}\}$, and similarly for $\{\bo_{k}\}\times\hat U_1^{n-k }$, then
\[\psi(\hat U_1^k,\bo_{n-k})=\phi^{-1}(W^u),\]
\[\psi(\bo_k,\hat U_1^{n-k})=\phi^{-1}(W^s).\]
Then there is a local chart near the critical point, $\Phi=\phi\circ\psi:\hat U\rightarrow X$, $\Phi(\bo)=p$, and
\[\Phi(\hat U_1^k,\bo_{n-k})=W^u,\]
\[\Phi(\bo_k,\hat U_1^{n-k})=W^s.\]
Hence we find a local chart $\Phi$ where we can always split the stable and unstable manifolds by coordinates, while maintaining the estimate of $(V-V_0)$ being flat.
\end{proof}

\noindent
Let us denote that $V=\nabla f=\lambda_ix_i\partial_i+O_i(\bx)\partial_i$, and $V_0=\lambda_ix_i\partial_i$, on a open bounded neighbourhood $U$ of $p$. Then we write a flow induced by $V$ as $G_t(\bx)$, and the corresponding standard flow by $F_t(\bx)$, namely,
\begin{align}
\nonumber &G_t(\bx)=\text{ solution to }
\left\{ {\begin{array}{ll}
\dint \bu(t)=V(\bu(t))_i=\lambda_iu_i(t)+O_i(\bu(t))\\
\bu(0)=\bx\\
\end{array} }, \right. \\
\nonumber &F_t(\bx)=\text{ solution to }
\left\{ {\begin{array}{ll}
\dint \bu(t)=V_0(\bu(t))_i=\lambda_iu_i(t)\\
\bu(0)=\bx\\
\end{array} }. \right.
\end{align}
Then solutions can be written down explicitly, or at least formally for $G$, as families of local diffeomorphisms near $p$ parametrized by $t$,
\begin{align}
\nonumber &F_t(\bx)_i=x_i e^{\lambda_i t},\\
\label{eqn_implicit_soln_of_G_t(x)} &G_t(\bx)_i= x_ie^{\lambda_i t}+e^{\lambda_i t}\int^t_0 e^{-\lambda_i s}O_i(G_s(\bx))\, ds,
\end{align}
where $i=1,\dots,n$.\\
With the help of $F_t(\bx)$ and $G_t(\bx)$, now we construct local coordinates where $V=V_0$ on the stable and unstable submanifolds.

\begin{proposition}\label{lem_v=v_0atWuWs}
Assume that Morse eigenvalues $\{\lambda_1,\dots,\lambda_n\}$ satisfy the $\mathbb{N}$-linearity condition, and $\lambda_1,\dots, \lambda_k>0$, $\lambda_{k+1},\dots,\lambda_n<0$. Let $W^u$ and $W^s$ be open neighbourhoods of the critical point $p$ on the stable and the unstable submanifolds of $p$, and let formal conjugate functions on $W^u$ and $W^s$ be
\begin{align}
\nonumber
 \Psi_u=\lim_{T\rightarrow \infty} F_{T}G_{-T} \text{ on } W^u,\\
\nonumber
 \Psi_s=\lim_{T\rightarrow \infty} F_{-T}G_{T} \text{ on } W^s.
\end{align}
Then $\Psi_u,\Psi_s$ are smooth local diffeomorphisms on $W^u$ and $W^s$ respectively. They fix the critical point, and
\[\Psi_u^*(V|_{W^u})=V_0|_{W^u}, \Psi_s^*(V|_{W^s})=V_0|_{W^s}.\]
In addition, there exist local coordinates on a small neighbourhood $U$ of $p$, induced by $\Psi_u$ and $ \Psi_s$, such that
\[|V-V_0|(\bx)=O(|\bx|^\infty),\]
\[V(\bx_k,\bo_{n-k})=V_0(\bx_k,\bo_{n-k}),\, V(\bo_k,\bx_{n-k})=V_0(\bo_k,\bx_{n-k}),\]
 for all $\bx, (\bx_k,\bo_{n-k}),(\bo_k,\bx_{n-k})\in U$
\end{proposition}
\noindent 
\begin{proof}
As established in Proposition \ref{lem_coordinate_out_of_WuWs}, there exist local coordinates $(x_1,\dots,x_n)$ where the critical point is mapped to $\bo$, with $W^u=\{(\bx_k,\bo)\}$, $W^s=\{(\bo,\bx_{n-k})\}$ and $|V-V_0|(\bx)=O(|\bx|^\infty)$. Denote this local coordinate chart explicitly by $\Phi:\hat U\rightarrow X$, namely, coordinates $(x_1,\dots,x_n)\in\hat U$, with $\hat U$ an open neighbourhood covering $\bo$ in $\RR^n$.\\
From now on, we restrict ourselves on $W^u$ within this local coordinate chart, and prove that $\Psi_u$ is a diffeomorphism. With this choice of coordinates,
\[G_t(\bx_k,\bo_{n-k})=(G_t(\bx_k,\bo_{n-k})_1,\dots,G_t(\bx_k,\bo_{n-k})_k,0,\dots,0),\]
and we write $G_t(\bx_k)=(G_t(\bx_k,\bo_{n-k})_1,\dots,G_t(\bx_k,\bo_{n-k})_k)$ in this proof.\\
The formal conjugate function $\Psi_u$ is
\[\Psi_u(\bx_k)=\lim_{T\rightarrow \infty} F_{T}G_{-T}(\bx_k)=\lim_{T\rightarrow \infty}\left(x_i+\int^{-T}_0 e^{-\lambda_i s}O_i(G_s(\bx_k))\, ds\right)_{i=1,\dots,k}.\]
By dominated convergence theorem, the following equation
\begin{equation}\label{eqn_l-driv_as_integration}
\nabla^l (\Psi_u(\bx_k)_i-x_i)=\lim_{T\rightarrow \infty}\int^{-T}_0 e^{-\lambda_i s}\nabla^l O_i(G_s(\bx_k))\, ds,
\end{equation}
 holds for all non-negative multi-indices $l$, so long as the right hand side limit exists for all $l$.\\
Taking into account that from the defining equation of $G_t(\bx)$,
\[\nabla^l \dint G_t(\bx_k)_i=\lambda_i \nabla^l G_t(\bx_k)_i+\nabla^l O_i(G_t(\bx_k)),\]
we have that for any $t\in \RR$ and multi-index $l$,
\begin{equation} \label{eqn_l-integrand_as_driv}
\dint (e^{-\lambda_i t}\nabla^l G_t(\bx_k))=e^{-\lambda_i t}\nabla^l O_i(G_t(\bx_k)).
\end{equation}
Take equation \eqref{eqn_l-integrand_as_driv} into \eqref{eqn_l-driv_as_integration}, so the $l$-th partial derivative of $\Psi_u(\bx_k)_i$ satisfies the relation
\begin{equation}\label{eqn_lim_of_psi_u}
\nabla^l (\Psi_u(\bx_k)_i)=\lim_{t\rightarrow\infty} e^{\lambda_i t} \nabla^l G_{-t}(\bx)_i,
\end{equation}
as long as the the limit on the right hand side of the equation exits.\\
To prove that $\Psi_u$ is a local diffeomorphism, it is enough to show that $\Psi_u$ is smooth and $D\Psi_u$ is invertible near $\bo_k$. Apparently, $D\Psi(\bx_k)|_{\bx_k=\bo_k}=Id_k$ is invertible. For smoothness, it is sufficient to prove that the right hand side of equation \eqref{eqn_lim_of_psi_u} is continuous for each $l$, and we will achieve that with decay estimates of $\bx_k$-derivatives of $G_t(\bx_k)$ when $t\rightarrow -\infty$.\\

\noindent
{\bf $C^0$ regularity of $\Psi_u$.}\\
To prove the existence of $\Psi_u$, we need decay estimate of $G_t(\bx)$ on $W^u$. Note that for every large enough integer $K$, there exists a positive constant $C_K$ such that
\[| O_i (\bx)|\leq C_K |\bx^{1/\lambda}|^K, \forall i=1, \dots,n,\]
 as $O_i(\bx)$ is locally flat for each $i$. Consider
 \begin{align}
 \nonumber \dint &(|G_t(\bx)_1|^{1/\lambda_1}+\dots+|G_t(\bx)_k|^{1/\lambda_k})\\
 \nonumber &=  |G_t(\bx)_1|^{1/\lambda_1}+\dots+|G_t(\bx)_k|^{1/\lambda_k}+ \sum_{i}\frac{sgn(G_t(\bx)_i)}{\lambda_i}G_t(\bx)_i^{1/\lambda_i-1} O_i(G_t(\bx))\\
 \nonumber &\geq |G_t(\bx)_1|^{1/\lambda_1}+\dots+|G_t(\bx)_k|^{1/\lambda_k} - C_1 |G_t(\bx)^{1/\lambda}|^{K-1+\max\limits_{1,\dots,k}\lambda_i}.
  \end{align}
Denote $u(t)=|G_t(\bx)_1|^{1/\lambda_1}+\dots+|G_t(\bx)_k|^{1/\lambda_k}=|G_t(\bx)^{1/\lambda}|$, $K_2:=K-1+\max\limits_{1,\dots,k}\lambda_i$, then the above inequality reduces to
\[\dint u(t)\geq u(t)-C_1 u(t)^{K_2},\]
which has explicit solution
\[\frac{u(t)^{K_2}}{u(t)-C_1u(t)^{K_2}}\leq C_2 e^{(K_2-1)t}\]
for some positive constant $C_2$. In addition, as we only care about $u(t)$ for $t\rightarrow -\infty$, the above inequality simplifies to
\[u(t)\leq \frac{C_2^{\frac{1}{K_2-1}}e^t}{(1+C_1C_2e^{(K_2-1)t})^{\frac{1}{K_2-1}}}.\]
An upshoot of this decaestimate is the observation
\[|\dint (e^{-\lambda_i t}G_t(\bx)_i)|=|e^{-\lambda_i t}O_i(G_t(\bx))|\leq C_3 e^{K\min \lambda_i t}.\]
And because the derivative is absolutely dominated by an integrable function, we know that indeed its intergation on the whole of $t\in(-\infty,0]$ exists, and consiquently so does $\Psi_u$, as
\[\Psi_u(\bx)_i=\int^{-\infty}_0 \dint (e^{-\lambda_i t}G_t(\bx)_i) dt.\]
Furthermore, the existence of this limit improves our initial decay estimate to
\[|G_t(\bx)_i|\leq Ce^{\lambda_i t}.\]

\noindent
{\bf $C^m$ regularity of $\Psi_u$.}\\
This will be done using bootstrapping method. Let the $(m-1)$-th induction hypothesis be that there exist a constant $C$ such that the $C^{m-1}$ norm of $G_t(\bx)_i$ has exponential decay, namely,
\[\|G_t(\bx)_i\|_{m-1}\leq Ce^{\lambda_i t}.\]
Then it will be sufficient to show that the corresponding $m$-th decay estimate holds.\\
We will utilize the following ODE
\[\dint \nabla^m G_t(\bx)_i=\lambda_i \nabla^m G_t(\bx)_i+\nabla^m (O_i(G_t(\bx)));\]
and with higher order chain rule
\[\nabla^m (O_i(G_t(\bx)))=\sum_{j=1,\dots,k}\frac{\partial O_i(G_t(\bx))}{\partial y_j}\nabla^mG_t(\bx)_j+R_m(O_i(G_t(\bx)),G_t(\bx)),\]
this ODE can be rewritten as
\begin{align*}
\dint \nabla^m G_t(\bx)_i=&\lambda_i \nabla^m G_t(\bx)_i\\
&+\sum_{j=1,\dots,k}\frac{\partial O_i(G_t(\bx))}{\partial y_j}\nabla^mG_t(\bx)_j+R_m(O_i(G_t(\bx)),G_t(\bx)),
\end{align*}
note that when $|m|=1$, the term $R_m(O_i(G_t(\bx)),G_t(\bx))$ vanishes, which otherwise contains derivatives of $G_t(\bx)$ of order lower than $m$. \\
Summing up $i=1,\dots, k$ gives
\begin{align*}
\dint |\nabla^m G_t(\bx)|\geq & m_\lambda |\nabla^m G_t(\bx)|\\
&- k|\nabla O(G_t(\bx))||\nabla^mG_t(\bx)|-\sum_i|R_m(O_i(G_t(\bx)),G_t(\bx))|,
\end{align*}
where $m_\lambda=\min_{i=1,\dots,k} \lambda_i$, and $|\nabla O(\bx)|=\sum_i|\nabla O_i(\bx)|$. For every large enough integer $K$, there exists a constant $C$ such that $|\nabla O(\bx)|\leq C|\bx|^K$, so there exists some $T\in (-\infty,0]$ such that whenever $t<T$,
\[m_\lambda-k|\nabla O(G_t(\bx))|\geq \frac{m_\lambda}{2}.\]
As a result,
\[\dint \left(e^{-\frac{m_\lambda}{2}t}|\nabla^m G_t(\bx)|\right)\geq -e^{-\frac{m_\lambda}{2}t}\sum_i|R_m(O_i(G_t(\bx)),G_t(\bx))|,\]
and this amounts to
\begin{align}\label{eqn_first_estimate_C^k}
\nonumber \lim_{t\rightarrow -\infty}&e^{-\frac{m_\lambda}{2}t}|\nabla^m G_t(\bx)|\\
 &\leq \int_{-\infty}^T e^{-\frac{m_\lambda}{2}t}\sum_i|R_m(O_i(G_t(\bx)),G_t(\bx))| dt+e^{-\frac{m_\lambda}{2}T}|\nabla^m G_T(\bx)|.
\end{align}
Note that using higher order chain rule,
\begin{align}
\nonumber |R_m(O_i(G_t(\bx)),G_t(\bx))|&\leq C\sum_{\substack{k_1+2k_2+\dots+(m-1) k_{m-1}=m,\\ k:=k_1+\dots+k_{m-1}}} |\nabla^k O_i(G_t(\bx))|\left\|G_t(\bx)\right\|_{m-1}^k\\
\nonumber &\leq C \|O(G_t(\bx))\|_m\, \max\{1, \|G_t(\bx)\|_{m-1}^m\},
\end{align}
so the integrand in \eqref{eqn_first_estimate_C^k} has exponential decay, following from that of $ \|O(G_t(\bx))\|_m$. Hence the integration converges, and there exists some constant $C$ such that
\[|\nabla^m G_t(\bx)|\leq Ce^{\frac{m_\lambda}{2}t}\]
for all $t<< T$.\\
Based on this estimate we can proceed with the following observation
\begin{align}
\nonumber &\left|\dint \left(e^{-\lambda_i t} \nabla^mG_t(\bx)_i\right)\right|\\
\nonumber &=\left|e^{-\lambda_i t}\left(\sum_{j=1,\dots,k}\frac{\partial O_i(G_t(\bx))}{\partial y_j}\nabla^mG_t(\bx)_j+R_m(O_i(G_t(\bx)),G_t(\bx))\right)\right|\\
\nonumber &\leq C_1 e^{C_2 t}
\end{align}
for some positive constants $C_1, C_2$. Because the derivative $\dint \left|e^{-\lambda_i t}\nabla^m G_t(\bx)_i\right|$ is absolutely dominated by an integrable function, itself is integrable over $t\in (-\infty,0]$. Hence $\nabla^m (\Psi_u)_i=\lim_{t\rightarrow -\infty}e^{-\lambda_i t}\nabla^m G_t(\bx)_i $ exists, and there exists a constant $C$ such that
\[\left|\nabla^m G_t(\bx)_i\right|\leq C e^{\lambda_i t}, \]
which is exactly what we want for the $m$-th induction hypothesis.\\
{\bf The local diffeomorphism generated by $\Psi_u, \Psi_s$.\\}
Given the construction as before, it is clear that $\Psi_u$ and $\Psi_s$ are $C^\infty$ diffeomorphisms on an open neighbourhood covering the origin of $\RR^k$ and $\RR^{n-k}$ respectively,and both of them fix the origin. Then there exist open neighbourhoods $\bo_k\in\hat U_u\subset \RR^k$ and $\bo_{n-k}\in \hat U_s\subset \RR^{n-k}$ such that $\hat U_k \times \hat U_s\subset \hat U$, and
\begin{align}
\nonumber \Psi=\Phi\circ(\Psi_u\times\Psi_s):&\hat U_k\times \hat U_s\longrightarrow X\\
\nonumber &(\bx_k,\bx_{n-k})\mapsto \Phi(\Psi_u(\bx_k),\Psi_s(\bx_{n-k})).
\end{align}
The diffeomorphism $\Psi$ is a local coordinate chart covering $p$, with $\Psi(\bo)=p$.\\
Apparently, over this coordinate chart, $\Psi(\hat U_u\times\{\bo_{n-k}\})=W^u, \Psi(\{\bo_k\}\times\hat U_s)=W^s$, and $V|_{W^u}=V_0|_{W^u}$, $V|_{W^s}=V_0|_{W^s}$. Moreover, as we start with the local coordinate $(\hat U, \Phi)$ where $|V-V_0|(\bx)=O(|\bx|^\infty)$, it is obvious that this local flatness of $(V-V_0)$ is preserved by the diffeomorphism $\Psi$ --- as the vanishing of all derivatives at the origin is passed along by chain rule.
\end{proof}
\subsubsection{Estimating $(V-V_0)$ with controlled error terms.}
\label{subsubsec3.1.3}
\begin{thm}\label{thm_estimate_V-V_0_WuWs}
Assume that the Morse eigenvalues $\lambda_1,\dots,\lambda_n$ are $\mathbb{N}$-linearly independent. Then on a small neighbourhood $\hat U$ containing $p$, for every large enough positive integer $\alpha$, there exist local coordinates $(x_1,\dots,x_n)$ with $p$ mapped to $\bo$, and a positive constant $C=C(\alpha)$ such that
\[|V-V_0|(\bx)\leq C|\bx_{k}|^{\alpha}|\bx_{n-k}|^{\alpha}\]
for every $\bx\in \hat U$.
\end{thm}
\begin{proof}
Observe that due to $V=V_0$ on the stable and unstable manifold (Proposition \ref{lem_v=v_0atWuWs}), combined with Proposition \ref{prop_lem_standardizing_on_WuWs}, the Taylor expansion of $V$ reduces to
\begin{equation}\label{eqn_assumption_in_L_V}
\CL_V x_i=V_i=\lambda_i x_i+
    \sum_{|\beta|\geq 2}C_{\beta,i}(\bx_k) \bx_{n-k}^\beta +
    \sum_{|\gamma|\geq 2}C_{\gamma,i}(\bx_{n-k}) \bx_{k}^\gamma + O(|\bx_k|^\infty|\bx_{n-k}^\infty|),
\end{equation}
where $\beta=(\beta_1,\dots,\beta_{n-k})$ and $\gamma=(\gamma_1,\dots,\gamma_k)$ are multi-indices,  $i=1\dots,n$, and $C_{\beta,i}(\bx_k)=O(|\bx_k|^\infty)$, $C_{\gamma,i}(\bx_{n-k})=O(|\bx_{n-k}|^\infty)$. This is true  on a small neighbourhood $U$ near the critical point. In addition,
\[\left\{
  \begin{array}{l}
    C_{\beta,i}(\bx_k)=\frac{1}{\beta!}\left(\left(\frac{\partial}{\partial \bx_{n-k}}\right)^\beta\CL_{V-V_0}x_i\right)(\bx_k,0)\\
    C_{\gamma,i}(\bx_{n-k}) =\frac{1}{\gamma!}\left(\left(\frac{\partial}{\partial \bx_{k}}\right)^\gamma\CL_{V-V_0}x_i\right)(0,\bx_{n-k}).\\
  \end{array}
\right.\]
Apparently, we only need $|\beta|,|\gamma|\geq \alpha$ instead of $\geq 1$ to conclude our proof. To achieve this, we will construct a sequence of local coordinates that get rid of the (m+1)-th order terms at m-th step:
\[{}^{0}x_i=x_i,\]
\begin{equation}\label{eqn_induction_on_x_i}
{}^{m+1}x_i={}^mx_i+
\left\{
  \begin{array}{l}
    \sum_{|\ba|=m+1}A_{\ba,i}({}^m\bx_{n-k})\,{}^m\bx_k^\ba,\ i=1,\dots,k, \\
    \sum_{|\bb|=m+1}B_{\bb,i}({}^m\bx_{k})\,{}^m\bx_{n-k}^\bb,\  i=k+1,\dots,n,  \\
  \end{array}
\right.
\end{equation}
where $\ba,\bb$ are multi-indices with k and (n-k) entries respectively.\\
The induction hypothesis is
\begin{equation}\label{eqn_indution_hypo}
\CL_V {}^mx_i= \lambda_i{}^mx_i+
\sum_{|\beta|\geq m+1}{}^mC_{\beta,i}({}^m\bx_k) {}^m\bx_{n-k}^\beta +
\sum_{|\gamma|\geq m+1}{}^mC_{\gamma,i}({}^m\bx_{n-k}) {}^m\bx_{k}^\gamma,
\end{equation}
Specifically, that ${}^mC_{\beta, i}({}^m\bx_k)$ and ${}^mC_{\gamma,i}({}^m\bx_{n-k})$ are locally flat is part of our assumption. \\
Then it is sufficient to prove the (m+1)-th hypothesis from the m-th. In fact, because ${}^mx_i$ is dominated by the linear part, all we need is that $\CL_V{}^{m+1}x_i-\lambda_i{}^{m+1}x_i$ has vanishing (m+1)-th order term, in terms of ${}^mx_i$.\\
Applying $\CL_V$ to ${}^{m+1}x_i$ in \eqref{eqn_induction_on_x_i} gives us
\begin{equation}\label{eqn_apply_assumption_to_induc}
\CL_V{}^{m+1}x_i= \CL_V {}^mx_i+\left\{
  \begin{array}{l}
    \sum_{|\ba|=m+1}\CL_V(A_{\ba,i}({}^m\bx_{n-k})\,{}^m\bx_k^\ba),\ i=1,\dots,k, \\
    \sum_{|\bb|=m+1}\CL_V(B_{\bb,i}({}^m\bx_{k})\,{}^m\bx_{n-k}^\bb),\  i=k+1,\dots,n.  \\
  \end{array}
\right.
\end{equation}
For simplicity we omit condition on $i$'s,
\begin{align}
\nonumber &\CL_V{}^{m+1}x_i-\lambda_i{}^{m+1}x_i\\
\nonumber &\overset{\eqref{eqn_apply_assumption_to_induc}, \eqref{eqn_induction_on_x_i}}=\CL_V {}^mx_i-\lambda_ix_i+\left\{
  \begin{array}{l}
    \!\sum_{|\ba|=m+1}\CL_V(A_{\ba,i}({}^m\bx_{n-k})\,{}^m\bx_k^\ba) -\lambda_iA_{\ba,i}({}^m\bx_{n-k})\,{}^m\bx_k^\ba \\
    \!\sum_{|\bb|=m+1}\CL_V(B_{\bb,i}({}^m\bx_{k})\,{}^m\bx_{n-k}^\bb) -\lambda_iB_{\bb,i}({}^m\bx_{k})\,{}^m\bx_{n-k}^\bb  \\
  \end{array}
\right.\\
\nonumber &\overset{\eqref{eqn_indution_hypo}}=
\sum_{|\beta|\geq m+1}{}^mC_{\beta,i}({}^m\bx_k) {}^m\bx_{n-k}^\beta +
\sum_{|\gamma|\geq m+1}{}^mC_{\gamma,i}({}^m\bx_{n-k}) {}^m\bx_{k}^\gamma \\
\nonumber &\indent+\left\{
  \begin{array}{l}
    \!\sum_{|\ba|=m+1}\CL_V(A_{\ba,i}({}^m\bx_{n-k})\,{}^m\bx_k^\ba) -\lambda_iA_{\ba,i}({}^m\bx_{n-k})\,{}^m\bx_k^\ba \\
    \!\sum_{|\bb|=m+1}\CL_V(B_{\bb,i}({}^m\bx_{k})\,{}^m\bx_{n-k}^\bb) -\lambda_iB_{\bb,i}({}^m\bx_{k})\,{}^m\bx_{n-k}^\bb.  \\
  \end{array}
\right.
\end{align}
Note that we only care about the vanishing $(m+1)$-order terms, so all higher orders are discarded. Also, using the assumption \eqref{eqn_indution_hypo}, we know that indeed $\CL_{V-V_0}x_i$ will at least add an order of $(m+1)$ to $\bx_k$ and $\bx_{n-k}$, rendering $\CL_{V-V_0}$ the higher order terms that we may omit. By doing this, we alter higher order coefficients at each step. Nevertheless, the local flatness of the coefficients is carried to the $(m+1)$-th step because this process is linear and finite in higher order terms. As a result,
\begin{align}
\nonumber &-\sum_{|\beta|\geq m+1}{}^mC_{\beta,i}({}^m\bx_k) {}^m\bx_{n-k}^\beta -
\sum_{|\gamma|\geq m+1}{}^mC_{\gamma,i}({}^m\bx_{n-k}) {}^m\bx_{k}^\gamma\\
\nonumber &=\left\{
  \begin{array}{l}
    \sum_{|\ba|=m+1}\CL_{V_0}(A_{\ba,i}({}^m\bx_{n-k}){}^m\bx_k^\ba) -\lambda_iA_{\ba,i}({}^m\bx_{n-k})\,{}^m\bx_k^\ba \\
    \sum_{|\bb|=m+1}\CL_{V_0}(B_{\bb,i}({}^m\bx_{k}){}^m\bx_{n-k}^\bb) -\lambda_iB_{\bb,i}({}^m\bx_{k})\,{}^m\bx_{n-k}^\bb, \\
  \end{array}
\right.
\end{align}
or more explicitly as
\begin{align}
\nonumber &\CL_{V_0}(A_{\ba,i}({}^m\bx_{n-k})) +(\sum_{j=1}^k a_j\lambda_j -\lambda_i) A_{\ba,i}({}^m\bx_{n-k})=-C_{\ba,i}({}^m \bx_{n-k}),\, i=1,\dots,k;\\
\nonumber &\CL_{V_0}(B_{\bb,i}({}^m\bx_{k})) +(\sum_{j=k+1}^n b_j\lambda_j-\lambda_i)B_{\bb,i}({}^m\bx_{k})=-C_{\bb,i}({}^m \bx_{k}),\, i=k+1,\dots,n.
\end{align}
Now we solve the first PDE at point $y_j(t)={}^m x_je^{\lambda_jt}, j=k+1,\dots,n$. From a simple calculation we learn that
\[\dint A_{\ba,i}(\by(t))=\CL_{V_0} A_{\ba,i}(\by(t)),\]
so the PDE reduces to
\[\dint A_{\ba,i}(\by(t))+ (\sum_{j=1}^k a_j\lambda_j -\lambda_i) A_{\ba,i}(\by(t))=-C_{\ba,i}(\by(t)).\]
Combined with the fact that $\by(t\rightarrow\infty)=\bo$ and $C_{\ba,i}(\bx_{n-k})=O(|\bx_{n-k}|^\infty)$, we have
\[A_{\ba,i}(\by(t))=e^{-(\sum_{j=1}^k a_j\lambda_j -\lambda_i)t}\int^\infty_0 e^{(\sum_{j=1}^k a_j\lambda_j -\lambda_i)s}C_{\ba,i}(\by(s))\, ds,\]
and the actual solution is given by the evaluation of $\by(t)$ at $0$, which is
\[A_{\ba,i}({}^m \bx_{n-k})=\int^\infty_0 e^{(\sum_{j=1}^k a_j\lambda_j -\lambda_i)s}C_{\ba,i}(\{{}^mx_je^{\lambda_js}\}_{j=k+1}^n)\, ds.\]
Because $C_{\ba,i}$ is flat at $\bo$, by dominated convergence, $A_{\ba,i}({}^m\bx_{n-k})$ exists and is smooth at least for all ${}^m\bx_{n-k}$ where $C_{\ba,i}({}^m\bx_{n-k})$ is controlled by a large enough power of ${}^m\bx_{n-k}$. Especially, if $(V-V_0)$ is supported on a sufficiently small neighbourhood containing $\bo$, which we later will require, then $A_{\ba,i}$ exists throughout the support of $(V-V_0)$.\\
Similarly, the solution to the second set of PDEs is
\[B_{\bb,i}({}^m \bx_{k})=-\int_{-\infty}^0 e^{(\sum_{j=k+1}^n a_j\lambda_j -\lambda_i)s}C_{\bb,i}(\{{}^mx_je^{\lambda_js}\}_{j=1}^{k})\, ds.\]
Repeating this induction process until $m=\alpha$, we then get the estimate $|V-V_0|(\bx)\leq C|\bx_{k}|^{\alpha}|\bx_{n-k}|^{\alpha}$ within a small neighbourhood $\hat U$, as claimed.
\end{proof}
\subsection{The contraction operator and conjugate functions.}\label{sec_contraction_operator_and_conjugate_functions}
An intuitive way of approaching the analytical Morse lemma is to seek for a dynamical system that relates a general Morse flow to its standard counterpart. The conjugate function $\lim_{T\rightarrow \infty}G_TF_{-T}$ that we used earlier on submanifolds, can be formally defined on the whole neighbourhood of $p$, and seemingly serve the propose of conjugating the generic flow with the standard one pretty well. The problem is that the existence of such a conjugation is far from straightforward, let alone its regularity. The idea behind our method here is motivated by the 1969 paper \cite{PalaisSmale69} by Palais and Smale, in which they discuss the existence of topological conjugation, which is they call structural stability, for a certain type of diffeomorphisms.\\
In this section, we shall give an analytical justification for the existence and regularity of these conjugate functions, and then construct the analytical Morse coordinates we want from them.\\
\\
From now on, we assume the following.\\
As we are only interested in the local behaviour of the flow lines, we will assume that $V=\nabla f=\sum_i\lambda_ix_i\partial_i+O_i(\bx)\partial_i$ on a open bounded coordinate neighbourhood $\mathfrak{U}$ of $p$, with $p$ mapped to $\bo$, and $V=V_0=\sum_i \lambda_ix_i\partial_i$ outside a compact neighbourhood $\mathfrak{U}_{V-V_0}$ containing $\mathfrak{U}$. In addition, $V$ is smooth on $\RR^n$. Such a vector field $V$ can be derived by the convolution $V=V_0+(\nabla f-V_0)*\phi$ with smooth bump function $\phi$ that is compactly supported and identically $1$ over $\mathfrak{U}$, with $\mathfrak{U}_{V-V_0}=\mathrm{support} (V-V_0)$.\\
Moreover, we shall assume that the conclusion of Theorem \ref{thm_estimate_V-V_0_WuWs}, namely, for all large enough positive integer $\alpha$ and constant $C=C(\alpha)$,
\[|V-V_0|(\bx)\leq C|\bx_{k}|^{\alpha}|\bx_{n-k}|^{\alpha},\]
holds for coordinates $\bx\in\mathfrak{U}_{V-V_0}$, with $p$ as $\bo$. As commented in the proof of the theorem, this is a reasonable requirement.
\subsubsection{Weighted Sobolev spaces on $(-\infty,0]$.}
\label{subsubsec3.2.1}
The following definitions of weighted Sobolev spaces are introduced by Lockhart and McOwen in their pioneer work \cite{Lockhart85} and \cite{Lockhart87}, which concern the definition and the analysis of weighted Sobolev spaces with finite open ends. By specifiying the base space to be the open ended line $(-\infty,0]$ in \cite[\S 3]{Lockhart87}, our notation is as follows.\\
Let $\bu\in C^\infty((-\infty,0];\RR^n)$, $p> 1$, $k\in \mathbb{N}$ and $\delta\in\RR$. The \textit{Sobolev norm} and the \textit{weighted Sobolev norm} of $\bu$ are
\[\|\bu\|_{p,k}=\sum_{j=0}^k\left[\int_{-\infty}^0\left|\left(\frac{d}{d t}\right)^j \bu(t)\right|^p dt \right]^{1/p},\]
\[\|\bu\|_{p,k,\delta}=\sum_{j=1}^k\left[\int_{-\infty}^0e^{-\delta p t} \left|\left(\frac{d}{d t}\right)^j \bu(t)\right|^p dt\right]^{1/p}.\]
The Sobolev space is defined as
\[L^p_k((-\infty,0];\RR^n)=\{\text{measurable function }\sigma:(-\infty,0]\rightarrow\RR^n: \|\sigma\|_{p,k}<\infty\},\]
and the local Sobolev space consists of all functions that belong to $L^p_k$ when restricted to any compact support, namely
\begin{align*}
L^p_{k,loc}((-\infty,0];\RR^n)=\{\text{measurable function }\sigma:(-\infty,0]\rightarrow\RR^n: \|\phi\sigma\|_{p,k}<\infty,\\
 \forall \phi\in C^\infty_0((-\infty,0];\RR)\}.
\end{align*}
It is easy to check that this definition coincides with the usual definition of Sobolev spaces and local Sobolev spaces. Then the \textit{weighted Sobolev space} over $(-\infty,0]$ with weight $\delta$ is
\[L^p_{k,\delta}((-\infty,0];\RR^n)=\{\bu\in L^p_{k,loc}((-\infty,0];\RR^n): \|\bu\|_{p,k,\delta}<\infty\}.\]
\subsubsection{The contraction operator $\CF$.}
\label{subsubsec3.2.2}
Let $u(t)=G_{t+T}F_{-T}(\bx)-F_{t}(\bx)$, $t\in(-\infty,0]$. Then we see that $u(t)\equiv 0$ for all $t<-T$, and $u(0)=G_TF_{-T}(\bx)-\bx$ resembles the conjugate function that we want. In addition, $u(t)$ is the solution to the ODE
\[\dint u(t)=V(u(t)+F_t(\bx))-V_0(F_t(\bx)), u(-\infty)=0.\]
This ODE has formal solution $u(t)=\int_{-\infty}^t V(u(t)+F_t(\bx))-V_0(F_t(\bx))\,ds$. Using this integration formula, we cook up an operator $\CF$,
\begin{align}
\nonumber&\CF: U\times L^p_{k,\delta}((-\infty,0];\RR^n)\rightarrow L^p_{k,\delta}((-\infty,0];\RR^n)\\
\label{eqn_functor_F}
&\CF(\bx,u(t))=\CF_\bx(u)(t)= \int_{-\infty}^t V(u(s)+F_s(\bx))\,ds-\int_{-\infty}^t V_0(F_s(\bx))ds\\
\nonumber &= \int_{-\infty}^t (V-V_0)(u(s) +F_s(\bx))\,ds+\int_{-\infty}^t V_0(u(s)+F_s(\bx))-V_0(F_s(\bx))ds,
\end{align}
where $U\subset \mathfrak{U}$ is an open coordinate chart containing the critical point, with $\bx\in U$ and the critical point $p$ corresponds to $\bo$, and for simplicity, we assume $U=B_\bo(R)$, the open ball of radius $R$ centred at $\bo$. 
\\
We claim that under suitable conditions, $\CF_\bx$ is indeed a contracting operator. To prove this, we begin with the fact that $\CF_\bx$ is shirking within its range.\\
Let $\Omega(r)$ be the closed convex domain $\Omega(r)\subset L^p_{k,\delta}((-\infty,0];\RR^n)$, defined by
\[\Omega=\{\gamma\in L^p_{k,\delta}:\|\gamma\|_{p,k,\delta}\leq r\},\]
for some constant $r$. This domain owns its convexity to that of the $L^p_{k,\delta}$-norm.
\begin{proposition}\label{lem_CF_shrink_in_range}
For suitably chosen $\delta,r>0$, there exists a positive real number $C<1$, such that
 \[\|\CF_\bx(\gamma_1)-\CF_\bx(\gamma_2)\|_{p,k,\delta}\leq C\|\gamma_1-\gamma_2\|_{p,k,\delta},\]
 for all $\gamma_1,\gamma_2\in \Omega(r)$ and all $\bx\in U$.
\end{proposition}
\noindent
To prove this, we will need the following lemma.
\begin{lem}\label{lem_integration_with_delta}
Given $\delta>0$, $k\in \mathbb{N}$. Let $w(t)\in L^p_{k,\delta}((-\infty,0],\RR^n)$
. Then for any $t\in(-\infty,0]$,
\[\|\int^s_{-\infty}w(u)\,du\|_{p,k,\delta;s\in(-\infty,t]}\leq \frac{C_0}{\delta}\|w(s)\|_{p,k,\delta;s\in(-\infty,t]},\]
where
 $C_0$ is a constant that only depends on $p$.
\end{lem}
\begin{proof}
Throughout this proof, let $w(t)\in L^p_{k,\delta}((-\infty,0],\RR^n)$ with $w(t)\equiv0$ for all $t<-T$, given some large enough $T$. Note that using the fact that $C^{\infty}_{cs}((-\infty, 0];\RR^n)$ is dense in $L^p_{k,\delta}((-\infty,0],\RR^n)$, this proof holds without assuming that $w(t)$ is compactly supported.\\
First let $k=0$.\\
It is sufficient to prove that for $w\in L^p_{0,\delta}((-\infty,t],\RR)$ with $w(t)\equiv0$ for all $t<-T$,  $\|\int^t_{-\infty}w(s)\,ds\|_{p,\delta}\leq \frac{C}{\delta}\|w(t)\|_{p,\delta}$.\\
Integrating by parts and using the fact that $w(t)\equiv0$ for all small enough $t$, we get
\begin{align} \nonumber&\|\int^s_{-\infty}w(u)\,du\|^p_{p,0,\delta;s\in(-\infty,t]}\\
\nonumber &= \int^t_{-\infty}e^{-\delta ps}ds\left|\int_{-\infty}^s w(u)du\right|^p\\
\nonumber &= -\left. \frac{e^{-\delta ps}}{\delta p}\left(\int_{-\infty}^s w(u)du\right)^p\right|_{s=-\infty}^{t}+ \int^t_{-\infty} \frac{e^{-\delta ps}}{\delta}w(s)\left(\int_{-\infty}^s w(u)du\right)^{p-1} ds\\
&\leq \underbrace{\frac{e^{-\delta pt}}{\delta p}\left|\int_{-\infty}^t w(u)du\right|^p}_{I}+ \underbrace{\int^t_{-\infty} \frac{e^{-\delta ps}}{\delta}w(s)\left(\int_{-\infty}^s w(u)du\right)^{p-1} ds}_{II}.\label{eqn_norm_p,0,delta}
\end{align}
For each of the two terms, we apply Cauchy-Schwartz inequality to get
\begin{align}
\nonumber I&=\frac{e^{-\delta pt}}{\delta p}\left|\int_{-\infty}^t e^{\delta u}e^{-\delta u}w(u)du\right|^p\\
\nonumber &\leq \frac{e^{-\delta pt}}{\delta p} \left[\left|\int_{-\infty}^t e^{\delta q u}du \right|^{1/q} \left|\int_{-\infty}^t e^{-\delta p u}|w(u)|^p du\right|^{1/p}\right]^{p}\\
\nonumber &=\frac{e^{-\delta pt}}{\delta p} \left[\left. \frac{e^{\delta qu}}{\delta q}\right|_{-\infty}^t\right]^{p/q} \int_{-\infty}^t e^{-\delta p u}|w(u)|^p du\\
\nonumber &=\frac{(p-1)^{p-1}}{\delta^p p^p}\int_{-\infty}^t e^{-\delta p u}|w(u)|^p du,
\end{align}
and
\begin{align}
\nonumber II&=\frac{1}{\delta}\int^t_{-\infty} e^{-\delta s}w(s)\left(e^{-\delta s}\int_{-\infty}^s w(u)du\right)^{p-1} ds\\
\nonumber &\leq \frac{1}{\delta} \left|\int_{-\infty}^t e^{-\delta p s}|w(s)|^p ds\right|^{1/p} \left|\int^t_{-\infty} e^{-\delta p s}\left|\int_{-\infty}^s w(u)du\right|^{p} ds\right|^{1-1/p}.
\end{align}
We relabel the two norms as following
\begin{align}
\nonumber &A:=\|\int^s_{-\infty}w(u)\,du\|_{p,0,\delta;s\in(-\infty,t]}= \left[\int^t_{-\infty}e^{-\delta ps}ds\left|\int_{-\infty}^s w(u)du\right|^p\right]^{1/p},\\
\nonumber &B:=\|w(s)\|_{p,0,\delta;s\in(-\infty,t]}=\left[\int_{-\infty}^t e^{-\delta p s}|w(s)|^p ds\right]^{1/p},
\end{align}
then the original inequality \eqref{eqn_norm_p,0,delta} simplifies to
\[A^p\leq \frac{(p-1)^{p-1}}{\delta^p p^p}B^p+\frac{1}{\delta}BA^{p-1}.\]
In non-trivial cases where $AB\neq 0$, we have
\[1\leq \frac{(p-1)^{p-1}}{ p^p}\left(\frac{B}{\delta A}\right)^p+\frac{B}{\delta A}.\]
Let $c$ be the positive solution to $1=\frac{(p-1)^{p-1}}{ p^p}c^p+c$, then the monotonicity of this polynomial in $\frac{B}{\delta A}$ gives $c\leq \frac{B}{\delta A}$. Let the constant $C_0=1/c$, which obviously only depends on $p$. As a result,
\[A\leq \frac{C_0}{\delta}B,\]
which is what we claim.\\
Secondly, $k>0$. By the definition of $\|\cdot\|_{p,k,\delta}$,
\begin{align}
\nonumber \Big\|\int^s_{-\infty} w(u)\,du\Big\|_{p,k,\delta}&=\sum_{j=0}^k \Big\|\left(\frac{d}{ds}\right)^j\int^s_{-\infty} w(u)\,du\Big\|_{p,\delta}\\
\nonumber & \leq \frac{C_0}{\delta}\sum_{j=0}^k\Big\|\left(\frac{d}{ds}\right)^jw(s)\Big\|_{p,\delta}= \frac{C_0}{\delta}\|w(s)\|_{p,k,\delta}.
\end{align}
\end{proof}

\begin{proof}[Proof of Proposition \ref{lem_CF_shrink_in_range}]
With the help of Lemma \ref{lem_integration_with_delta}, we can control $\CF_\bx(\gamma_1)-\CF_\bx(\gamma_2)$ with its integrand, as it is an integration:
\begin{align}
\nonumber &\|\CF_\bx(\gamma_1)-\CF_\bx(\gamma_2)\|_{p,k,\delta}\\
\nonumber &= \Big\|\int_{-\infty}^t (V-V_0)(\gamma_1(s) +F_s(\bx))-(V-V_0)(\gamma_2(s) +F_s(\bx))\,ds\\
\nonumber&\indent+\int_{-\infty}^t V_0(\gamma_1(s)+F_s(\bx))-V_0(\gamma_2(s)+F_s(\bx))ds \Big\|_{p,k,\delta}\\
\nonumber &\leq \frac{C_0}{\delta}\|(V-V_0)(\gamma_1(s) +F_s(\bx))-(V-V_0)(\gamma_2(s) +F_s(\bx))\|_{p,k,\delta}\\
&\indent +\frac{C_0}{\delta}\|V_0(\gamma_1(s)+F_s(\bx))-V_0(\gamma_2(s) +F_s(\bx))\|_{p,k,\delta}. \label{eqn_contracting_parts}
\end{align}
As $V_0$ is globally linear, we fully understand that the second term of \eqref{eqn_contracting_parts}, which is controlled by
\[\frac{C_0}{\delta}\|V_0(\gamma_1(s)+F_s(\bx))-V_0(\gamma_2(s) +F_s(\bx))\|_{p,k,\delta} \leq\frac{C_0}{\delta}\max\limits_{i=1,\dots,n}\{|\lambda_i|\} \|\gamma_1-\gamma_2\|_{p,k,\delta}.\]
Hence the problem amounts to finding an upper bound for the first term,  $\left(\|(V-V_0)(\gamma_1(s) +F_s(\bx))-(V-V_0)(\gamma_2(s) +F_s(\bx))\|_{p,k,\delta}\right)$.\\
Here are two helpful observations which we will use extensively in bounding the first term. One observation is the \emph{chain rule for higher order derivatives}
\begin{equation}\label{eqn_general_chain_rule}
\left(\dint\right)^jf(g(t))=\sum_{l_1,\dots,l_j}\frac{j!}{l_1!\cdots l_j!} \nabla^lf(g(t))\left(\frac{g^{(1)}(t)}{1!}\right)^{l_1}\cdots \left(\frac{g^{(j)}(t)}{j!}\right)^{l_j},
\end{equation}
where $j=l_1+2l_2+\dots+jl_j$ with the summation done over all such possible $l_i$'s, and $l:=l_1+\dots+l_j$. And the other observation is the pointwise control of the difference of two products,
\begin{align}
\nonumber |&x_1\cdots x_n-y_1\cdots y_n|\\
& \leq \sum_{i=1}^n \|x_1,y_1\|_{\infty}\cdots \|x_{i-1},y_{i-1}\|_{\infty} |x_i-y_i| \|x_{i+1},y_{i+1}\|_{\infty}\cdots \|x_{n},y_{n}\|_{\infty},\label{eqn_difference_of_products}
\end{align}
which can be proven by a simple argument of induction.\\
Now, let us examine the first term, which is
\begin{align}
\nonumber&\|(V-V_0)(\sigma_1)-(V-V_0)(\sigma_2)\|_{p,k,\delta} \\ \nonumber&\indent =\sum_{j=0}^k\left\|\left(\dint\right)^j \left((V-V_0)(\sigma_1)-(V-V_0)(\sigma_2)\right)\right\|_{p,\delta},
\end{align}
with $\sigma_i=\gamma_i(s)+F_s(\bx)$, and we start with a pointwise estimate for each $t$-derivative. For $j>1$, we have
\begin{align}
\nonumber &\left|\left(\dint\right)^j \left((V-V_0)(\sigma_1)-(V-V_0)(\sigma_2)\right)\right|\\
\nonumber &\leq \Big| \sum_{l_1,\dots,l_j} \frac{j!}{l_1!\cdots l_j!} \frac{1}{(1!)^{l_1}\cdots (j!)^{l_j}}\big\{\nabla^l(V-V_0)(\sigma_1)[\sigma_1^{(1)}(t)]^{l_1}\cdots [\sigma_1^{(j)}(t)]^{l_j}\\
\nonumber & \indent - \nabla^l(V-V_0)(\sigma_2)[\sigma_2^{(1)}(t)]^{l_1}\cdots [\sigma_2^{(j)}(t)]^{l_j}\big\}\Big| \\
\nonumber &\leq \sum_{l_1,\dots,l_j} \frac{j!}{l_1!\cdots l_j!} \frac{1}{(1!)^{l_1}\cdots (j!)^{l_j}} \\
\nonumber &\Big\{|\nabla^l(V-V_0)(\sigma_1(t))-\nabla^l(V-V_0)(\sigma_2(t))|\  \|\sigma_1^{(1)},\sigma_2^{(1)}\|_\infty^{l_1}\cdots \|\sigma_1^{(j)},\sigma_2^{(j)}\|_\infty^{l_j}\\
\nonumber & + \sum_{i=1}^j \|\nabla^l(V-V_0)(\sigma_1),\nabla^l(V-V_0)(\sigma_2)\|_\infty \|\sigma_1^{(1)},\sigma_2^{(1)}\|_\infty^{l_1} \cdots |\sigma_1^{(i)}(t)-\sigma_2^{(i)}(t)|\\
\nonumber & \cdot l_i\|\sigma_1^{(i)}-\sigma_2^{(i)}\|_\infty^{l_i -1} \cdots \|\sigma_1^{(j)},\sigma_2^{(j)}\|_\infty^{l_j}
\Big\},
\end{align}
combining constants and taking $\|\sigma\|_{\infty,j}=\sum_{i=0}^j\|\sigma^{(i)}\|_{\infty}$,
\begin{align}
\nonumber & \leq \sum_{l_1,\dots,l_j} \frac{j!}{l_1!\cdots l_j!} \frac{1}{(1!)^{l_1}\cdots (j!)^{l_j}} \Big\{ \|\nabla^{l+1}(V-V_0)\|_\infty \|\sigma_1,\sigma_2\|_{\infty,j}^l |\sigma_1(t)-\sigma_2(t)|\\
\nonumber &\indent +\sum_{i=1}^j \|\nabla^{l}(V-V_0)\|_\infty \|\sigma_1,\sigma_2\|_{\infty,j}^{l-1}\cdot l_i|\sigma_1^{(i)}(t) -\sigma_2^{(i)}(t)|\Big\}
\end{align}
and denoting $|\sigma(t)|_{*,j}=\sum_{i=0}^j|\sigma^{(i)}(t)|$ gives
\begin{align*}
\nonumber  \leq \sum_{l_1,\dots,l_j} \frac{j!}{l_1!\cdots l_j!} \frac{1}{(1!)^{l_1}\cdots (j!)^{l_j}} &\Big\{\|\nabla^{l+1}(V-V_0)\|_\infty \|\sigma_1,\sigma_2\|_{\infty,j}^l \\
\nonumber &+l\|\nabla^{l}(V-V_0)\|_\infty \|\sigma_1,\sigma_2\|_{\infty,j}^{l-1}\Big\}\cdot |\sigma_1(t)-\sigma_2(t)|_{*,j}.
\end{align*}
Given this pointwise estimate, we now place a bound on the corresponding $p,k,\delta$-norm term,
\begin{align}
\nonumber &\|(V-V_0)(\sigma_1)-(V-V_0)(\sigma_2)\|_{p,k,\delta}\\
\nonumber &\leq \sum_{j=1}^k \sum_{l_1,\dots,l_j} \frac{j!}{l_1!\cdots l_j!} \frac{1}{(1!)^{l_1}\cdots (j!)^{l_j}} \Big\{\|\nabla^{l+1}(V-V_0)\|_\infty \|\sigma_1,\sigma_2\|_{\infty,j}^l +l\|\nabla^{l}(V-V_0)\|_\infty\\
\nonumber &\indent \cdot \|\sigma_1,\sigma_2\|_{\infty,j}^{l-1}\Big\}\cdot \|\sigma_1(t)-\sigma_2(t)\|_{p,k,\delta} + \|\nabla(V-V_0)\|_\infty \|\sigma_1(t)-\sigma_2(t)\|_{p,k,\delta}\\
\nonumber &\leq C_a\|V-V_0\|_{\infty,k} \, \max\{\|\sigma_1,\sigma_2\|_{\infty,k}^k,1\} \, \|\sigma_1-\sigma_2\|_{p,k,\delta},
\end{align}
where $C_a$ is an algebraic constant that only depends on $k$. We may also assume that $p$ is big enough in our case, so that by Sobolev embedding, $\|\sigma_i\|_{\infty,k}\leq C_S(\|\gamma_i\|_{p,k,\delta}+R)$, where $R$ is the radius of domain $U\subset \RR^n$, and $C_S$ is the Sobolev constant. As a result,
\begin{align*}
\nonumber &\|\CF_\bx(\gamma_1)-\CF_\bx(\gamma_2)\|_{p,k,\delta}\\
\nonumber &\leq \frac{C_0}{\delta}\left[C_a\|V-V_0\|_{\infty,k} \, \max_{i=1,2}\{C_S^k(\|\gamma_i\|_{p,k,\delta}+R)^k,1\}+\max_{i=1,\dots,n}|\lambda_i| \right] \|\gamma_1-\gamma_2\|_{p,k,\delta}\\
&\leq \frac{C_0}{\delta}\left[C_a\|V-V_0\|_{\infty,k} \, \max\{C_S^k(r+R)^k,1\}+\max_{i=1,\dots,n}|\lambda_i| \right] \|\gamma_1-\gamma_2\|_{p,k,\delta}\\
\nonumber & = C\|\gamma_1-\gamma_2\|_{p,k,\delta},
\end{align*}
with $C=\frac{C_0}{\delta}\left[C_a\|V-V_0\|_{\infty,k} \, \max\{C_S^k(r+R)^k,1\}+\max_{i=1,\dots,n}|\lambda_i| \right]$. \\
To force this constant $C<1$, it is sufficient to require $\delta$ to be large. As a result, $\CF$ is a contraction mapping.
\end{proof}
\noindent
We will assume that $\delta$ is fixed from now on.\\
The following lemma guarantees the choice of nice constants such that the zero function does not get ``drifted'' too far away by $\CF$.
\begin{proposition}\label{lem_CF_drifting_domain}
Given $\delta>0$, $r>0$ such that the positive contraction constant $C<1$ is chosen as before,
 then
\[\|\CF_\bx(0)(t)\|_{p,k,\delta}\leq (1-C)r,\]
for all $\bx\in U$.
\end{proposition}
\begin{proof}
Recall that
\[\CF_\bx(0)(t)=\int^t_{-\infty} V(F_s(\bx)-V_0(F_s(\bx))ds.\]
Then Lemma \ref{lem_integration_with_delta} and chain rule combined tell us
\begin{align}
\nonumber &\|\CF_\bx(0)(t)\|_{p,k,\delta}\leq \frac{C_0}{\delta} \sum_{j=0}^k\Big\|\left(\dint\right)^j(V-V_0)(F_t(\bx))\Big\|_{p,\delta}\\
\nonumber &\leq \frac{C_0}{\delta} \sum_{j=0}^k \sum_{l_1,\dots,l_j} \frac{j!}{l_1!\cdots l_j!} \frac{1}{(1!)^{l_1}\cdots (j!)^{l_j}} \|\nabla^l(V-V_0)(F_s(\bx))\mathbf{\lambda}^j F_s(\bx)^l\|_{p,\delta},
\end{align}
where $\lambda^j$ and $F_s(\bx)^l$ should be seen as $n$-vectors with powers taking on each elements, e.g. $\lambda^j=(\lambda_1^j,\dots, \lambda_n^j)$.\\
If eigenvalues $\lambda_i$'s satisfy $\mathbb{N}$-linearity condition, then by Theorem \ref{thm_estimate_V-V_0_WuWs}, we have the pointwise assumption for $V-V_0$,
\[|\nabla^j(V-V_0)(\bx)|\leq C_2(|x_1|^{1/\lambda_1}+\dots+|x_k|^{1/\lambda_k})^\alpha (|x_{k+1}|^{1/|\lambda_{k+1}|}+\dots+|x_n|^{1/|\lambda_n}|)^\alpha,\]
for any given positive constants $\alpha$ and its associated constant $C_2=C_2(\alpha)$. Based on this assumption we have
\begin{align}
\nonumber &|\nabla^l(V-V_0)(F_s(\bx))\lambda^jF_s(\bx)^l|\leq C_2 n\cdot \max_{i=1,\dots,n}|\lambda_i|^j \cdot\\
\nonumber &\,\cdot(|x_1|^{1/\lambda_1}+\dots+|x_k|^{1/\lambda_k})^\alpha (|x_{k+1}|^{1/|\lambda_{k+1}|}+\dots+|x_n|^{1/|\lambda_n}|)^\alpha \sum_{i=1,\dots,n}|x_ie^{\lambda_it}|^l.
\end{align}
Hence,
\begin{align}
\nonumber & \|\nabla^l(V-V_0)(F_s(\bx))\lambda^jF_s(\bx)^l\|_{p,\delta}\leq C_2 n \cdot\max_{i=1,\dots,n}|\lambda_i|^j\cdot(|x_1|^{1/\lambda_1}+\!\dots\!+|x_k|^{1/\lambda_k})^\alpha\\
\nonumber &\indent \cdot  (|x_{k+1}|^{1/|\lambda_{k+1}|}+\!\dots\!+|x_n|^{1/|\lambda_n}|)^\alpha R^l \left(\frac{e^{-\delta p T}-1}{\delta p}\right)^{1/p}\!.
\end{align}
Given that $T$ is defined as satisfying
\[|x_1|^2e^{-2\lambda_1T}+\dots+|x_n|^2e^{-2\lambda_nT}=R^2,\]
we have the estimate
\[e^T\leq (R(|x_{k+1}|+\dots+|x_n|)^{-1})^{\frac{1}{\min\limits_{k+1,\dots,n}|\lambda_i|}} .\]
Combining this estimate with the conclusion of Theorem \ref{thm_estimate_V-V_0_WuWs},
\begin{align*}
\nonumber &\|\CF_\bx(0)\|_{p,k,\delta} \\
\nonumber &\leq \frac{C_0}{\delta} \Big\{ n C_2 C_a \max_{i,j}|\lambda_i|^j \max\{R,\dots,R^k\} (|x_1|^{1/\lambda_1}+\!\dots\!+|x_k|^{1/\lambda_k})^\alpha (|x_{k+1}|^{1/|\lambda_{k+1}|}\\
\nonumber & \, +\!\dots\!+|x_n|^{1/|\lambda_n|})^\alpha\cdot (\frac{1}{\delta p} (R^{\frac{\delta p} {\min\limits_{k+1,\dots,n}|\lambda_i|}} (|x_{k+1}|+\dots+|x_n|)^{-\frac{\delta p}{\min\limits_{k+1,\dots,n}|\lambda_i|}}-1) )^{1/p}\Big\}\\
&\leq \frac{C_0}{\delta} n C_2 C_a \max_{i,j}|\lambda_i|^j \max\{R,\dots,R^k\}\, |\bx^{1/\lambda}_k|^\alpha|\bx^{1/\lambda}_{n-k}|^\alpha |\bx_{n-k}|^{\frac{-\delta} {\min\limits_{k+1,\dots,n}|\lambda_i|}}\\
\nonumber &\leq \frac{C_0}{\delta} n C_2 C_a \max_{i,j}|\lambda_i|^j \max\{R,\dots,R^k\} R^{\frac{2\alpha-\delta} {\min\limits_{k+1,\dots,n}|\lambda_i|}}.
\end{align*}
To achieve $\|\CF_bx(0)\|_{p,k,\delta}\leq (1-C)r$, we want to employ parameters that satisfy the following inequality
\begin{align*}
\|&\CF_\bx(0)\|_{p,k,\delta} \leq \frac{C_0}{\delta} n C_2 C_a \max_{i,j}|\lambda_i|^j \max\{R,\dots,R^k\} R^{\frac{2\alpha-\delta} {\min\limits_{k+1,\dots,n}|\lambda_i|}} \\
&\leq \left(1-\frac{C_0}{\delta}\left[C_a\|V-V_0\|_{\infty,k} \, \max\{C_S^k(r+R)^k,1\}+\max_{i=1,\dots,n}|\lambda_i| \right]\right)r= (1-C)r.
\end{align*}
It becomes clear that by choosing $\alpha$ large enough and $R>0$ sufficiently small, our estimate in concern $\|\CF(0)\|_{p,k,\delta}\leq (1-C)r$ is achieved. Note that the when $R$ is reasonably close to zero, the contraction constant $C$ only gets smaller, resulting in the improvement of right hand side of the inequality. So the choice of $\alpha$ and the reduction in the radius $R$ of the region $U$, or even altering for a larger $\delta$, will not jeopardise our estimate of $C$, even if the constant $C$ depends on $R$ and $\delta$.\\
As a final comment, in terms of the behaviour of $\CF_\bx(0)$ for every fixed $\bx$, the estimate gets improved near the stable and the unstable manifold as well, for $|\bx_k||\bx_{n-k}|$ is small.
\end{proof}
\begin{thm}\label{thm_CF_is_contracting_on_Omega}
Assume that the set of eigenvalues $\lambda_1,\dots,\lambda_n$ satisfies the $\mathbb{N}$-linearity condition. Given properly chosen constants $\delta>0$, $k\in \mathbb{N}$, $R>0$ and $r>0$, there exists a positive constant $C<1$, and $\Omega=\Omega(r)\subset L^p_{k,\delta}$, such that
\[\CF:U\times \Omega\rightarrow \Omega\]
is a contraction operator with contracting constant $C$.
\end{thm}
\begin{proof}
From Proposition \ref{lem_CF_drifting_domain}, it is straightforward to see that for all $\bx\in U$, the zero function isn't mapped too far away, i.e.
\[\|\CF_\bx(0)(t)\|_{p,k,\delta}\leq (1-C)r.\]
In addition, Proposition \ref{lem_CF_shrink_in_range} shows that for all $\gamma_1,\gamma_2\in \Omega$,
\[\|\CF_\bx(\gamma_1)-\CF_\bx(\gamma_2)\|_{p,k,\delta}\leq C\|\gamma_1-\gamma_2\|_{p,k,\delta}.\]
especially for all $\gamma\in \Omega$,
\[\|\CF_\bx(\gamma)(t)-\CF_\bx(0)(t)\|_{p,k,\delta}\leq Cr.\]
then we have $\|\CF_\bx(\gamma)\|\leq r$, namely $\CF_\bx(\Omega)\subset \Omega$. As a result, $\CF$ is a contraction mapping on $\Omega$.
\end{proof}
\subsubsection{Proof of Theorem \ref{thm_conjugate_function_existence}.}\label{subsubsec_proof_of_main}
\noindent Now we present a proof for the main theorem.
\begin{proof}[Proof of Theorem \ref{thm_conjugate_function_existence}]
Theorem \ref{thm_CF_is_contracting_on_Omega} establishes the fact that $\CF:U\times \Omega\rightarrow \Omega$ is a contracting operator on $\Omega$, which is a convex subset of the complete metric space $L^p_{k,\delta}$ and hence $\Omega$ also is a convex complete metric space. By contraction mapping theorem, there exists a unique fixed point $p(t, \bx)\in \Omega$ of $\CF$, such that
\[p(t,\bx)=\int_{-\infty}^t (V-V_0)(p(s,\bx)+F_t(\bx))ds + \int_{-\infty}^t V_0(p(s,\bx)+F_s(\bx))-V_0(F_s(\bx))ds.\]
In fact, the integration indicates that $p(t,\bx)$ is smooth in $t$. And all $\delta>0$ large enough is a suitable choice, as the proof of Lemma \ref{lem_integration_with_delta} shows. Hence $p(t,\bx)\rightarrow \bo$ when $t\rightarrow -\infty$, for every fixed $\bx\in U$.\\
Let $q(t,\bx)=p(t,\bx)+F_t(\bx)$, then it solves $\dint u(t)=V(u(t))$, with its initial condition $u(0)=\Phi(\bx)$ determined by $p(t,\bx)$. In other words, $q(t,\bx)=G_t(\Phi(\bx))$, and
\begin{equation}\label{eqn_p_and_Phi}
p(t,\bx)=G_t(\Phi(\bx))-F_t(\bx).
\end{equation}
As $p(t,\bx)$ is the difference of two flow lines, both of which is smooth in $t$, it is legitimate to fix $t=0$ in equation \eqref{eqn_p_and_Phi} and write
\[\Phi(\bx)=p(0,\bx)+\bx.\]
Now let us determine the explicit form of $\Phi(\bx)$. Denote the generalized real number $T(\bx, F)=\inf\{t\in \RR: F_{-t}(\bx)\notin U\}$ for any $\bx\in U$, which is the time of the flow $F_{-t}(\bx)$ when it first exists the domain $U$; and similarly $T(\bx,G)$ be the time parameter of $G_{-t}(\bx)$ exiting the domain $U$. It follows that $T(\bx,F)$ equals $+\infty$ if and only if $\bx\in W^u(F)$, and when $\bx\notin W^u(F)$, $T(\bx,F)\in \RR$.\\
Let $\hat T=\max \{T(\bx, F),T(\Phi(\bx),G)\}$.\\
If $\hat T$ is finite, then for any $t<-\hat T$,
\[p(t,\bx)=F_{t+\hat T}(G_{-\hat T}(\Phi(\bx)))-F_t(\bx)=F_{t+\hat T}(G_{-\hat T}(\Phi(\bx))-F_{-\hat T}(\bx)),\]
as $G$ and $F$ have identical flow lines outside of $U$, and $F_t(\bx)$ is linear with respect to $\bx$. Consequently, $p(t,\bx)$ is a flow line of $F$, and $p(t,\bx)\rightarrow \bo$ when $t\rightarrow -\infty$, namely $p(t,\bx)$ is on the unstable submanifold, and it is contained in $L^p_{k,\delta}((-\infty,0];\RR^n)$ for all sufficiently large $\delta$. Such a flow line can only be the static flow at the critical point. Hence in this case
\[G_{-\hat T}(\Phi(\bx))-F_{-\hat T}(\bx)=\bo,\]
\[\Phi (\bx)=G_{\hat T}(F_{-\hat T}(\bx)))=\lim_{t\rightarrow \infty}G_tF_{-t}(\bx),\]
and $\Phi $ is a conjugate function between the local standard and general flow lines,
\[G_t(\Phi(\bx))=\Phi(F_t(\bx)), \forall t\in \RR,\]
\[\Phi^*(V)=V_0.\]
Otherwise, $\hat T$ is infinite. This implies that at least one of $\bx$ and $\Phi(\bx)$ is on the unstable submanifold of $F$, which is also the unstable submanifold of $G$. In fact, both of them are in the unstable locus: If not, let's say that $\bx\in W^u$ but $\Phi(\bx)\notin W^u$, then for sufficiently small $T$, $\|G_{T}(\Phi(\bx))\|$ is at least bounded by the radius of $U$, while $F_T(\bx)$ can be made as close to $\bo$ as we wish; this contradicts $p(t,\bx)\rightarrow \bo$, so both $\bx$ and $\Phi(\bx)$ are in $W^u$. But as we established before, $G_t(\bx)=F_t(\bx)$ for any $\bx$ in the stable and unstable loci, so we can always write
\[p(t,\bx)=F_t(\Phi(\bx)-\bx).\]
Again using the fact that an unstable flow line that is an element in $L^p_{k,\delta}((-\infty,0];$ $\RR^n)$ for all large $\delta$ is indeed trivial, we have in this case
\[\Phi(\bx)=\bx,\]
for all $\bx\in W^u$, and it indeed is the conjugate function between $F_t(\bx)$ and $G_t(\bx)$, which in this case are identical. And trivially, it still holds that
\[\Phi^*(V)=V_0.\]
\newline
To construct local coordinates by means of conjugation, it is necessary for $\Phi(\bx)$ to have enough regularity. It turns out that this is guaranteed by the Implicit Function Theorem (IFT) for Banach spaces \cite[10.2.1, 10.2.3]{Dieudonne69}, which we recall as follows :
\begin{addmargin}[1em]{.5em}
Let $E,F,G$ be Banach spaces, $f$ a continuously differentiable mapping of an open subset $A$ of $E\times F$ into $G$. Let $(x_0,y_0)$ be a point of $A$ s.t. $f(x_0,y_0)=0$ and that the partial derivative $D_2f(x_0,y_0)$ be a linear homeomorphism of $F$ onto $G$. Then there is an open neighbourhood $U_0$ of $x_0$ in E such that for every open connected neighbourhood $U$ contained in $U_0$, there is a unique continuous mapping $u$ of $U$ into $F$ such that $u(x_0)=y_0$, $(x,u(x))\in A$ and $f(x,u(x))=0$ for any $x\in U$. Furthermore, $u$ is continuously differentiable in $U$, and its derivative is
\[u'(x)=-(D_2f(x,u(x)))^{-1}\circ(D_1f(x,u(x))).\]
If in addition $f$ is $p$ times continuously differentiable in a neighbourhood of $(x_0,y_0)$, then $u$ is $p$ times continuously differentiable in a neighbourhood of $x_0$.
\end{addmargin}
Let Banach spaces $E=\RR^n$, $F,G=L^p_{k,\delta}((-\infty,0];\RR^n)$, and the open subset $A=U\times \Omega(r+\epsilon)^\circ$, with $\CF:U\times \Omega(r)\rightarrow \Omega(r)$ being contracting and $\epsilon$ small enough. Consider the operator
\begin{align}
\nonumber &\CP: E\times F\rightarrow G\\
\nonumber &\CP(\bx,u(t))=\CF_\bx(u(t))-u(t).
\end{align}
Then what we established before translates to that $(x_0,y_0)=(\bx, p(t,\bx))$ is the unique solution to $\CP(x_0,y_0)=0$. Moreover, if $D_2\CP(\bx,p(t,\bx))$ is indeed a linear homeomorphism from $L^p_{k,\delta}((-\infty,0];\RR^n)$ onto itself, then IFT gives us the regularity of $p(t,\bx)$ with respect to $\bx$ -- because the uniqueness of the fixed point guarantees that $p(t,\bx)$ coincides with the local inverse function. Consequently, the fact that $\CP$ is $C^\infty$ in $(\bx, u(t))$ implies that $p(t,\bx)$ is $C^\infty$ in $\bx$. Specifically, by $\Phi(\bx)=p(0,\bx)+\bx$, the smoothness of $\Phi(\bx)$ follows. Combining the smoothness of $\Phi$ with the fact that $\Phi(\bx)=\bx$ on the stable and the unstable manifold, which includes the origin, $\Phi$ is smooth and invertible on a small neighbourhood of $\bo$, and this makes $\Phi$ a local diffeomorphism.\\
Now we show that $D_2\CP(\bx,p(t,\bx))$ is a surjective linear homeomorphism. Let us denote $A(\cdot)=D_2\CP(\bx,p(t,\bx))(\cdot)$ to be a linear operator from $L^p_{k,\delta}((-\infty,0];\RR^n)$ to itself. If $\delta u(t)$ is a small variation that is sufficiently smooth, then we have
\begin{align}
\nonumber &\CP(\bx,p(t,\bx)+\delta u(t))-\CP(\bx,p(t,\bx))=A(\delta u(t))+o(\delta u(t))\\
\nonumber &=\int_{-\infty}^t [V(p(t,\bx)+\delta u(t)+F_s(\bx))-V(p(t,\bx)+F_s(\bx))]ds -\delta u(t)+o(\delta u(t))\\
\nonumber &=\int_{-\infty}^t DV(p(t,\bx)+F_s(\bx))\cdot \delta u(s) ds-\delta u(t)+o(\delta u(t))\\
\nonumber &=\int_{-\infty}^t \Hess f(G_s(\Phi(\bx)))\cdot \delta u(s) ds-\delta u(t)+o(\delta u(t)).
\end{align}
As a result, the derivative of $\CP$ is
\[A(v(t))=D_2\CP(\bx,p(t,\bx))(v(t))=\int_{-\infty}^t \Hess f(G_s(\Phi(\bx)))\cdot v(s) ds-v(t).\]
Apparently $A$ is linear. Moreover, there exists a positive constant $c$, which we may assume less than 1, such that
\[ \|(A+Id)(v(t))\|_{p,k,\delta}\leq c\|v(t)\|_{p,k,\delta}.\]
This is because
\begin{align}
\nonumber\|(A+Id)(v(t))\|_{p,k,\delta}&=\|\int_{-\infty}^t \Hess f(G_s(\Phi(\bx)))\cdot v(s) ds\|_{p,k,\delta}\\
\nonumber &\leq \frac{C_0}{\delta}\|\Hess f(G_s(\Phi(\bx)))\cdot v(s) \|_{p,k,\delta}\\
\nonumber &\leq \frac{C_0}{\delta}\|\Hess f\|_\infty\|v(t)\|_{p, k,\delta},
\end{align}
 following from Lemma \ref{lem_integration_with_delta}; and in addition we claim that $c:=\frac{C_0}{\delta}\|\Hess f\|_\infty<1$, which can be easily achieved by picking larger $\delta$ in Lemma \ref{lem_integration_with_delta}'s proof. It follows that $A$ is a continuous linear map on $L^p_{k,\delta}$ and reasonably close to the identity. Then this operator is an isomorphism, and its inverse can be constructed as
 \[A^{-1}= -(Id-(A+Id))^{-1}=-\sum_{k=0}^\infty (A+Id)^k,\]
 which is the unique limit of an absolutely convergent sequence in the dual space of $L^p_{k,\delta}$. Hence $D_2\CP(\bx,p(t,\bx))$ is a homeomorphism and onto, as the IFT requires, which then concludes our proof.

\end{proof}

\noindent Mathematical Institute, Radcliffe Observatory Quarter, Oxford, UK. OX2 6GG\\
Email: Yixuan.Wang1@maths.ox.ac.uk
\end{document}